\documentclass[a4paper,11pt,reqno]{article}
\usepackage{relsize}
\usepackage{cite}
\usepackage{color}
\usepackage{mathtools}
\usepackage[dvipsnames]{xcolor}

\usepackage{hyperref, enumitem}

\hypersetup{
  colorlinks   = true, 
  urlcolor     = blue, 
  linkcolor    = Purple, 
  citecolor   = red 
}
\usepackage{amsmath,amsthm,amssymb}

\usepackage[margin=2.5cm]{geometry}

\usepackage{graphicx}
\usepackage[T1]{fontenc}
\usepackage{authblk}
\usepackage[english]{babel}
\usepackage{stackrel}

\theoremstyle{definition}
\newtheorem{theorem}{Theorem}[section]

\newtheorem{proposition}[theorem]{Proposition}
\newtheorem{corollary}[theorem]{Corollary}
\newtheorem{lemma}[theorem]{Lemma}
\newtheorem{problem}[theorem]{Problem}
\newtheorem{definition}[theorem]{Definition}
\newtheorem{example}[theorem]{Example}
\newtheorem{remark}[theorem]{Remark}

\newtheorem{conjecture}{Conjecture}

\newcommand{\F}{\mathbb{F}}
\newcommand{\Z}{\mathbb{Z}}
\newcommand{\M}{\mathcal{M}}

\DeclareMathOperator{\wt}{wt}
\DeclareMathOperator{\supp}{supp}
\DeclareMathOperator{\Char}{Char}
\DeclareMathOperator{\PG}{PG}
\DeclareMathOperator{\Gr}{Gr}

\newcommand{\qbinom}[2]{\genfrac{[}{]}{0pt}{}{#1}{#2}_{q}}

\title{{\bf Extremal character spectra of finite projective systems}}

\author{Shitao Li\thanks{School of Internet, Anhui University, Hefei, Anhui 230039, China.
E-mail: \textsf{lishitao0216@163.com}.},~ 
 Yang Li\thanks{School of Physical and Mathematical Sciences, Nanyang Technological University, Singapore 637371, Singapore. E-mail: \textsf{li-y@ntu.edu.sg}. },~
 Gaojun Luo\thanks{School of Mathematics, Nanjing University of Aeronautics and Astronautics, Nanjing, Jiangsu 211106, China. E-mail: \textsf{gaojun\_luo@nuaa.edu.cn}}}

\begin{document}
\maketitle

\begin{abstract}
Let \(\mathcal M\) be a finite spanning multiset of points in \(\PG(k-1,q)\), and let \(\Char_{\mathcal M}\) be its character function. For \(1\le r\le k\), we study the maximum number of distinct values attained by \(\Char_{\mathcal M}\) on the codimension-\(r\) flats under a total-multiplicity budget. The spectrum is subject to two natural upper bounds, whose extremal cases lead respectively to \(r\)-maximum character spectrum (\(r\)-MCS) and \(r\)-full character spectrum (\(r\)-FCS) systems. We establish the existence of \(r\)-FCS systems in a broad parameter range, and solve the size-constrained extremal problem completely when \(r=k-1\), including the minimum-cost multiplicity profiles. For \(1\le r<k\), put \(M_r=\qbinom{k}{r}\) and \(B_r=\qbinom{k-1}{r}\). The minimum size \(\nu_r(k,q)\) of an \(r\)-MCS system satisfies
\[
\left\lceil\frac{\binom{M_r}{2}}{B_r}\right\rceil
\le \nu_r(k,q)
\le \binom{M_r-1}{2}+\frac{q^k-q^{k-r}}{q-1}.
\]
The upper bound is at most \(\binom{M_r}{2}\), strictly less when \(r\le k-2\), and exact when \(r=k-1\). Its proof applies positive additive labelings to the complement of one flat, whose trace is empty while all remaining traces are distinct. For \(r=1\), we improve the known upper bounds on the minimum length of such systems. We also develop global moment identities and local flat-incidence bounds that impose strong structural constraints on large character spectra. Under the projective-system correspondence, complementary character spectra are the support-weight spectra of \(r\)-dimensional subcodes.
\end{abstract}



\section{Introduction}
A classical problem in finite geometry is to study point sets, multisets, or families of subspaces through their incidences with subspaces of prescribed dimension. Blocking sets, designs, covering problems, intersection theorems, and resolving systems are classical examples of this general principle; see, for example, \cite{ABDN2024,ABNR2022,Ball2012,BBDM2023,BT2026,BDGP2024,FW1981,KSS2025,Katona1966}.
In this paper, we investigate a different but closely related extremal problem for such incidence structures.
Rather than asking whether an intersection is nonempty or spanning, we assign nonnegative integer multiplicities to the points of a finite projective space and consider the resulting weighted incidence values on all flats of a fixed codimension. Therefore, our objective is not a single extremal intersection size, but the number of distinct values attained by the corresponding weighted incidence map.

Let \(q\) be a prime power, let \(k\ge1\), and put \(V=\F_q^k\) and \(\Pi=\PG(k-1,q)\). Let \(\mathcal M=(M,m)\) be a spanning multiset in \(\Pi\), where \(m:M\to\Z_{>0}\) assigns a positive integer multiplicity to each point of \(M\).
Its character function is defined by
\[\Char_{\mathcal M}(A)=\sum_{P\in A\cap M}m(P),\quad A\subseteq\Pi.
\]
Write \(|\M|=\Char_{\mathcal M}(\Pi)\) for the total multiplicity of \(\mathcal M\). We refer to such a spanning multiset as a projective \([|\M|,k]\) system.
For \(1\le r\le k\), let \(\mathcal F_r(\Pi)\) denote the family of codimension-\(r\) flats of \(\Pi\), with the convention that \(\mathcal F_k(\Pi)=\{\varnothing\}\). Consider the character spectrum and its complementary spectrum
\[\Sigma_r(\mathcal M)=\bigl\{\Char_{\mathcal M}(\Lambda):\Lambda\in\mathcal F_r(\Pi)\bigr\}
,\quad
\Omega_r(\mathcal M)=\bigl\{|\M|-\Char_{\mathcal M}(\Lambda):\Lambda\in\mathcal F_r(\Pi)\bigr\}.
\]
Since the two spectra have the same cardinality, our extremal problem is to determine how many distinct values the character function of a projective system can assume on \(\mathcal F_r(\Pi)\). Let
\(\ell_r(\mathcal M)=|\Sigma_r(\mathcal M)|=|\Omega_r(\mathcal M)|\)
be the size of the two spectra. For \(n\ge k\), let 
\[
L_r(n,k,q)
=
\max\{
\ell_r(\mathcal M):
\mathcal M\text{ is spanning and }|\mathcal M|\le n
\}.
\] 
Thus \(L_r(n,k,q)\) is the extremal function under a total-multiplicity budget \(n\).

The terminology is inherited from the coding-theoretic correspondence. 
Choose representatives of the points of \(\mathcal M\), repeated according to their multiplicities, as the columns of a \(k\times |\M|\) matrix, and let \(C_{\mathcal M}\) be the code generated by its rows.
Orthogonal complementation identifies the values in \(\Omega_r(\mathcal M)\) with the support weights of the \(r\)-dimensional subcodes of \(C_{\mathcal M}\). This correspondence, together with the relevant coding terminology, is established in Subsection~\ref{subsec:coding-correspondence}. For \(r=1\), the corresponding maximum-weight and full-weight spectrum problems were studied in \cite{SZSC,A2019,AN2019,CT2018}. In addition, generalized Hamming weights and higher support distributions were investigated in \cite{Wei1991,Klove1992,Britz2002,Britz2007}. The geometric interpretation of small weight spectra provides a related perspective. Calderbank and Kantor \cite{CK1986} survey the connections between two-weight codes, projective two-intersection sets, and strongly regular graphs. Here we instead seek the largest possible number of distinct incidence values and allow flats of arbitrary fixed codimension.

The minimum-distance parameter of a projective \([n,k,d]\) system depends only on its largest hyperplane character and is given by
\(d=n-\max\{\Char_{\mathcal M}(H):H\text{ is a hyperplane of }\Pi\}.\)
Our problem is higher-order in two senses: hyperplanes are replaced by flats of arbitrary prescribed codimension, and the largest character is replaced by the entire image of the character function. Fixed-rank layers of the subspace lattice belong to Delsarte's regular-semilattice framework \cite{Delsarte1976}, while multiplicity-sensitive covering questions for subspaces have also been studied from an extremal viewpoint \cite{BBDM2023}. Here one integral weighting is fixed, and the objective is the cardinality of its image on the whole Grassmannian.
From a linear-algebraic viewpoint, the characters are the coordinates of the product of a point--flat incidence matrix with a nonnegative integer multiplicity vector. The ranks of finite-geometric incidence and inclusion matrices were studied by Kantor \cite{Kantor1972} and Frumkin and Yakir \cite{FY1990}. Our extremal parameter concerns the number of distinct output coordinates, rather than the rank of the incidence transformation.

For every projective system \(\mathcal M\), there are two universal upper bounds on \(\ell_r(\mathcal M)\). The domain of the character function has \(|\mathcal F_r(\Pi)|=\qbinom{k}{r}\) elements. Moreover, if \(\Lambda=\PG(U)\in\mathcal F_r(\Pi)\), the images in \(V/U\) of representatives of the occurrences outside \(\Lambda\) span this \(r\)-dimensional quotient. Hence at least \(r\) occurrences lie outside \(\Lambda\), so \(\Omega_r(\mathcal M)\subseteq\{r,\ldots,|\M|\}\). Consequently, we have \(\ell_r(\M)
\le\min\left\{\qbinom{k}{r},|\M|-r+1\right\}.\) This implies that 
\begin{equation}\label{eq:intro-two-bounds}
L_r(n,k,q)\le\min\left\{\qbinom{k}{r},\,n-r+1\right\}
\end{equation}
The equality cases of these two bounds lead to distinct extremal classes.

\begin{definition}\label{def:character-extremes}
A finite projective \([|\M|,k]\) system \(\mathcal M\) is said to have an \emph{\(r\)-maximum character spectrum}, or to be \(r\)-MCS, if \(\ell_r(\mathcal M)=\qbinom{k}{r}\). Equivalently, \(\Char_{\mathcal M}\) is injective on \(\mathcal F_r(\Pi)\). It is said to have an \emph{\(r\)-full character spectrum}, or to be \(r\)-FCS, if \(\ell_r(\mathcal M)=|\M|-r+1\). Equivalently, \(\Omega_r(\mathcal M)=\{r,r+1,\ldots,|\M|\}\).
\end{definition}
The two upper bounds in \eqref{eq:intro-two-bounds} lead to two complementary extremal problems. The first is to determine those values of \(n\) for which the interval bound \(L_r(n,k,q)=n-r+1\) is attained, equivalently, for which there exists an \(r\)-FCS projective \([n,k]\) system. The second is to determine the smallest \(n\) for which the Grassmannian bound \(L_r(n,k,q)=\qbinom{k}{r}\) is attained. 
For later reference, we introduce the corresponding threshold.
We denote by
\begin{equation}\label{def:nu}
\nu_r(k,q)=\min\left\{n\ge k:L_r(n,k,q)=\qbinom{k}{r}\right\}
\end{equation}
the minimum size of an \(r\)-MCS projective system. This minimum is finite. Indeed, assigning the multiplicities \(2^0,2^1,2^2,\ldots\) to the points of \(\Pi\), enumerated in any order, as in the binary construction of \cite[Theorem~3.4]{AN2019}, gives a spanning projective multiset whose character is injective on all projective subspaces: distinct subspaces have distinct point sets, and binary expansions are unique. Theorem~\ref{thm:intro-threshold} below gives much smaller \(r\)-MCS systems.
Note that \(L_r(n,k,q)\) is defined by allowing projective systems of total multiplicity at most \(n\). Hence it is useful to distinguish the two equality cases. 
If an \(r\)-FCS projective system \(\mathcal M\) realizes \(L_r(n,k,q)=n-r+1\), then necessarily
\(|\M|=n\). Indeed, \(\ell_r(\mathcal M)=|\mathcal M|-r+1\le n-r+1,\) and the equality forces \(|\mathcal M|=n\). 
By contrast, an \(r\)-MCS projective system attaining \(L_r(n,k,q)=\qbinom{k}{r}\) may have total multiplicity strictly smaller than \(n\). However, at the threshold \(n=\nu_r(k,q)\), every such system has size exactly \(\nu_r(k,q)\); otherwise an \(r\)-MCS projective system of smaller size would contradict the minimality of \(\nu_r(k,q)\).

For convenience, write \(M_r=\qbinom{k}{r}\), \(B_r=\qbinom{k-1}{r}\), \(H_r=\binom{M_r}{2}\), and \(\vartheta_k=(q^k-1)/(q-1)\). Our first theorem gives a sufficient condition for the interval bound to be attained and determines the budget-constrained extremal function exactly in the point case.

\begin{theorem}\label{thm:intro-fixed-size}
Let \(1\le r\le k\le n\).
\begin{enumerate}
\item[(i)] If \(n-r+1<2^{k-r+1}\), then \(L_r(n,k,q)=n-r+1\). Equivalently, there exists an \(r\)-FCS projective \([n,k]\) system.

\item[(ii)] Suppose \(k\ge2\) and \(r=k-1\). Then
\[L_{k-1}(n,k,q)=\max\{m:2\le m\le\vartheta_k,\ \mu_k(m)\le n\},\]
where \(\mu_k(m)=\binom{m}{2}+\max\{0,k-m+1\}\). Consequently, \(\nu_{k-1}(k,q)=\binom{\vartheta_k}{2}\), and a \((k-1)\)-FCS projective \([n,k]\) system exists if and only if \(n\in\{k,k+1\}\).
\end{enumerate}
\end{theorem}

Part (i) is obtained by a direct-sum lifting of the ordinary full weight spectrum theorem. We conjecture that its numerical condition is also necessary, and we prove this under a natural common-flat hypothesis. Part (ii) reduces the point case to an exact integer optimization in which the cost of creating distinct positive multiplicities is separated from the additional cost of spanning the ambient projective space.
In the following, we develop global moment identities and local flat-incidence bounds, imposing strong structural constraints on large character spectra.

\begin{theorem}\label{thm:intro-moments}
Let \(1\le r\le k\), let \(\mathcal M\) be a projective \([N,k]\) system, and label its point occurrences by \(P_1,\ldots,P_N\). For \(J\subseteq[N]\), let \(\rho_{\mathcal M}(J)\) be the dimension of the span of nonzero representatives of the selected occurrences, with \(\rho_{\mathcal M}(\varnothing)=0\). Then, for every integer \(t\ge0\),
\begin{equation}\label{eq:binomial-moments}
 \sum_{\Lambda\in\mathcal F_r(\Pi)}\binom{\Char_{\mathcal M}(\Lambda)}{t} =
 \sum_{\substack{J\subseteq[N]\\ |J|=t}}\qbinom{k-\rho_{\mathcal M}(J)}{r}.
\end{equation}
If \(1\le r\le k-1\) and \(t\ge1\), then
\begin{equation}\label{eq:moment-image-bound}
 \binom{\ell_r(\mathcal M)}{t+1}\le\sum_{\substack{J\subseteq[N]\\ |J|=t}}
 \qbinom{k-\rho_{\mathcal M}(J)}{r}\le B_r\binom{N}{t}.
\end{equation}
\end{theorem}

The first moment controls its total incidence mass, and the second separates parallel pairs of occurrences from rank-two pairs. We complement these identities by a local invariant, the maximum character on a flat, which transfers restrictions on the represented matroid and on finite projective arcs into spectral upper bounds.
The final part separates the geometric problem of distinguishing flats from the arithmetic cost of assigning multiplicities.

\begin{definition}\label{def:separating-labeling}
Let \(1\le r\le k-1\). A set \(S\subseteq\Pi\) is \emph{\(r\)-separating} if the traces \(S\cap\Lambda\), \(\Lambda\in\mathcal F_r(\Pi)\), are pairwise distinct. For such a set, define
\[\tau_r(S)=\min\left\{\sum_{P\in S}a_P: a_P\in\Z_{>0},\quad
 \Lambda\longmapsto\sum_{P\in S\cap\Lambda}a_P \text{ is injective on }\mathcal F_r(\Pi)\right\}.
\]
\end{definition}

Every \(r\)-separating set in this range spans \(\Pi\), as proved in Proposition~\ref{prop:separating-spans}. The minimum defining \(\tau_r(S)\) is attained, since successive powers of two distinguish the distinct traces. The following theorem combines this cost with explicit choices of support to bound \(\nu_r(k,q)\).

\begin{theorem}\label{thm:intro-threshold}
Let \(1\le r\le k-1\).
\begin{enumerate}
\item[(i)] The minimum size of an \(r\)-MCS system satisfies
\[
\nu_r(k,q)=\min\{\tau_r(S):S\subseteq\Pi\text{ is \(r\)-separating}\}.
\]
\item[(ii)] We have
\begin{equation}\label{eq:threshold-summary}
\left\lceil\frac{H_r}{B_r}\right\rceil
\le\nu_r(k,q)
\le \binom{M_r-1}{2}+\frac{q^k-q^{k-r}}{q-1}
\le H_r.
\end{equation}
The last inequality is strict when \(r\le k-2\), and the upper bound is attained when \(r=k-1\). Consequently, \(\nu_r(k,q)=H_r\) if and only if \(r=k-1\).
\end{enumerate}
\end{theorem}

The lower bound in \eqref{eq:threshold-summary} follows from the first moment and is at least \(M_r+r-1\), as shown in \eqref{eq:lower-bound-comparison}. The second moment can exclude this first-moment endpoint. 
For the upper bounds, Theorem~\ref{thm:positive-labeling} applies the initial-label edge-discrimination lemma of Bhattacharya, Das, and Ganguly \cite[Lemma~1]{BDG2014}. It gives positive labels of total cost at most \(|S|+\binom h2\), where \(h\) is the number of nonempty traces on a separating set \(S\). On the full point set, subtracting one from every label preserves distinct characters and gives \(\nu_r(k,q)\le H_r\). To improve this bound, we take \(S\) to be the complement of one codimension-\(r\) flat. This set has \((q^k-q^{k-r})/(q-1)\) points and exactly \(M_r-1\) nonempty traces, which yields the sharper estimate in \eqref{eq:threshold-summary}. Thus the auxiliary labeling construction leads directly to bounds for \(\nu_r(k,q)\). For \(r=1\), an \(r\)-MCS system is the projective counterpart of a maximum weight spectrum code \cite{SZSC,AN2019}, and the upper bound in \eqref{eq:threshold-summary} reads \(\nu_1(k,q)\le\binom{\vartheta_k-1}{2}+q^{k-1}\). This quantity is of order \(q^{2k-2}\) and improves the known upper bounds on the minimum length of maximum weight spectrum codes; see Corollary~\ref{cor:mws-length} and Remark~\ref{rem:mws-comparison}.

The paper is organized as follows. 
In Section~\ref{sec:preliminaries}, we record the projective-system and coding-theoretic background and the separating-set property needed later.
In Section~\ref{sec:fixed-length}, we prove Theorem~\ref{thm:intro-fixed-size} and develop the moment identities in Theorem~\ref{thm:intro-moments} and the local incidence bounds. In Section~\ref{sec:mcs-threshold}, we study the \(r\)-MCS threshold, with lower bounds and endpoint obstructions in Subsection~\ref{subsec:threshold-lower} and the labeling constructions proving Theorem~\ref{thm:intro-threshold} in Subsection~\ref{subsec:threshold-upper}. Finally, in Section~\ref{sec:open}, we conclude the paper and present several open problems.

\section{Preliminaries}\label{sec:preliminaries}
Throughout the paper, \(q\) is a prime power, \(V=\F_q^k\), and \(\Pi=\PG(V)=\PG(k-1,q)\). All vector-space dimensions are over \(\F_q\). For integers \(a\ge0\), ordinary binomial coefficients \(\binom ab\) are zero when \(b<0\) or \(b>a\). 
We retain the notation introduced in the Introduction.

\subsection{Projective multisets and character spectra}\label{subsec:projective-systems}

A projective space of dimension \(k-1\) over the finite field \(\mathbb{F}_q\) is defined as
\[\operatorname{PG}(k-1,q):=\bigl(\mathbb{F}_q^k\setminus\{0\}\bigr)/\sim,\]
where \(u\sim v\) if \(u=\lambda v\) for some nonzero
\(\lambda\in\mathbb{F}_q\). The equivalence class that a nonzero vector \(v\) belongs to is denoted by \([v]\). The \(1\)-dimensional, \(2\)-dimensional, \(\ldots\), \((k-1)\)-dimensional vector subspaces of \(\mathbb{F}_q^k\) correspond to the points, lines, \(\ldots\),
hyperplanes of \(\Pi\).
In fact, we will often identify the points of \(\Pi\) with the \(1\)-dimensional vector subspaces of \(\mathbb{F}_q^k\). Note that while the dimension of a projective subspace is one less than the dimension of its corresponding vector subspace, the codimensions remain the same.
For a subset \(S\) of points in \(\Pi\), the subspace formed by taking the linear span of \(S\) is denoted by \(\langle S\rangle\). 

If \(U\le \F_q^k\) has vector dimension \(s\ge1\), then \(\PG(U)=\{P\in\Pi:P\le U\}\), which is naturally isomorphic to \(\PG(s-1,q)\). We set \(\PG(\{0\})=\varnothing\). For \(0\le s\le k\), let \(\Gr_q(s,k)\) be the Grassmannian of \(s\)-dimensional subspaces of \(V\). Its cardinality is the Gaussian coefficient
\[\qbinom{k}{s}=\frac{(q^k-1)(q^{k-1}-1)\cdots(q^{k-s+1}-1)}{(q^s-1)(q^{s-1}-1)\cdots(q-1)},\]
with the usual convention that the empty product is \(1\). For integer parameters, we put \(\qbinom ab=0\) when \(a<0\), \(b<0\), or \(b>a\). With this notation, we identify \(\mathcal F_r(\Pi)\) with \(\{\PG(U):U\in\Gr_q(k-r,k)\}\), where \(\mathcal F_k(\Pi)=\{\varnothing\}\). For \(j\ge0\), write
\(\theta_q(j)=\frac{q^{j+1}-1}{q-1},\) and set \(\theta_q(-1)=0\). In particular,
\(\vartheta_k=\theta_q(k-1)=\frac{q^k-1}{q-1}\) is the number of points of \(\Pi\). We refer to \cite{Casse2011} for further details on finite projective spaces.

Let \(\mathcal M=(M,m)\) be a finite multiset of points in \(\Pi\), where \(m(P)\) denotes the multiplicity of \(P\in M\). We extend \(m\) to all points of \(\Pi\) by setting \(m(P)=0\) for \(P\notin M\). Its total multiplicity and character function are
\[|\mathcal M|=\sum_{P\in\Pi}m(P),\quad
\Char_{\mathcal M}(A)=\sum_{P\in A}m(P)\quad(A\subseteq\Pi).\]
The support of \(\mathcal M\) is \(\supp(\mathcal M)=\{P\in\Pi:m(P)>0\}=M\). We call \(\mathcal M\) a \emph{projective system} if \(M\) spans \(\Pi\), and a \emph{simple projective system} if
\(m(P)\le1\) for every \(P\in\Pi\). 
Two projective systems \((M,m)\) and \((M',m')\) are \emph{projectively equivalent} if a projectivity induced by an invertible linear map on \(V\) sends one to the other and preserves multiplicities. Equivalently, there exists \(\varphi\in\operatorname{PGL}(k,q)\) such that \(m'(\varphi(P))=m(P)\) for every \(P\in\Pi\), after extending both multiplicity functions by zero. If \(\Lambda\in\mathcal F_r(\Pi)\), then its complementary character is
\[|\mathcal M|-\Char_{\mathcal M}(\Lambda)=\sum_{P\in\Pi\setminus\Lambda}m(P).\]
The sets \(\Sigma_r(\mathcal M)\), \(\Omega_r(\mathcal M)\), and the quantity \(\ell_r(\mathcal M)\) are as defined in the Introduction. We shall also use \(d_r(\mathcal M)=\min\Omega_r(\mathcal M)\). To distinguish the spectrum from the full distribution, put
\(A_{r,j}(\mathcal M)=|\{\Lambda\in\mathcal F_r(\Pi):
\Char_{\mathcal M}(\Lambda)=j\}|\) for \(j\ge0\).
Thus \(\Sigma_r(\mathcal M)=\{j:A_{r,j}(\mathcal M)>0\}\), whereas \(A_{r,j}\) records how often each character occurs.
For \(r=1\), this recovers the usual minimum-distance parameter of a projective system: a projective \([n,k,d]_q\) system is precisely a spanning projective multiset \(\mathcal M\) of size \(n\) with \(d=d_1(\mathcal M)\).
Besides the abbreviations \(M_r\) and \(B_r\) introduced in the Introduction, we put
\(B_r^{(2)}=\qbinom{k-2}{r}.\)
Thus \(B_r\) is the number of members of \(\mathcal F_r(\Pi)\) through a fixed point, whereas \(B_r^{(2)}\) is the number through a fixed projective line.

\subsection{The coding-theoretic correspondence}\label{subsec:coding-correspondence}

The formulation of our extremal problem is projective-geometric.
For completeness, we record here its exact correspondence with support weights of linear codes.
Let \(\M\) be a projective \([n,k]\) system, and list its point occurrences as \(P_1,\ldots,P_n\), where repeated points are listed according to their multiplicities. Choose a nonzero representative \(d_i\in P_i\) for each occurrence. Then the matrix
\[
G_{\mathcal M}=(d_1\mid\cdots\mid d_n)\in\F_q^{k\times n}
\]
has row rank \(k\) and is called a \emph{matrix realization} of \(\mathcal M\). Conversely, every \(k\times n\) matrix of row rank \(k\) with no zero column determines a projective system by projectivizing its columns. Permuting the columns or multiplying individual columns by nonzero field elements does not change the projective system.

Write \([n]=\{1,\ldots,n\}\). For a labeled set \(J\subseteq[n]\), define
\[\rho_{\mathcal M}(J)=\dim_{\F_q}\operatorname{span}_{\F_q}\{d_j:j\in J\},
\quad\rho_{\mathcal M}(\varnothing)=0.\]
The rank is independent of the chosen representatives. The rank function on the ground set \([n]\) defines the matroid represented by \(G_{\mathcal M}\): a set \(J\) is independent if \(\rho_{\mathcal M}(J)=|J|\), and a \emph{circuit} is a nonempty inclusion-minimal dependent set. Its \emph{girth} \(\gamma(\mathcal M)\) is the smallest circuit size, with girth \(\infty\) if no circuit exists. The matroid depends only on the labeled occurrences; a column permutation merely relabels its ground set. There are no loops, that is, circuits of size one, because all columns are nonzero. Two occurrences form a parallel pair precisely when they are supported at the same projective point. A matroid is \(q\)-representable if its rank function arises from a matrix over \(\F_q\). It is uniform of rank \(k\) if its independent sets are exactly the subsets of size at most \(k\).

We now isolate the coding-theoretic interpretation of the geometric model, extending the standard projective-system dictionary \cite{Greene1976,ABN2022}. Let \(G_{\mathcal M}\) be a matrix realization, and let \(C_{\mathcal M}\le\F_q^n\) be its row space. For \(x\in\F_q^n\), put \(\supp(x)=\{i:x_i\ne0\}\). If \(D\le C_{\mathcal M}\), define
\[
\supp(D)=\bigcup_{x\in D}\supp(x),
\quad
\wt(D)=|\supp(D)|.
\]
The integer \(\wt(D)\) is the \emph{support weight} of \(D\). Orthogonal complements in \(V=\F_q^k\) are taken with respect to the standard product.

The following standard correspondence is the higher-dimensional version of the usual projective-system description of Hamming weights; see, for example, \cite{Wei1991,Klove1992,Britz2007}.

\begin{proposition}\cite{Wei1991,Klove1992,Britz2007}\label{prop:char-r}
Let \(1\le r\le k\), let \(\mathcal M\) be a projective \([n,k]\) system, and let \(W\in\Gr_q(r,k)\). Put \(D_W=\{xG_{\mathcal M}:x\in W\}\). Then \(\dim D_W=r\) and
\[\wt(D_W)=n-\Char_{\mathcal M}\bigl(\PG(W^\perp)\bigr).\]
Consequently,
\[\begin{aligned}
\Omega_r(\mathcal M)
 &=\{\wt(D):D\le C_{\mathcal M},\ \dim D=r\},\\
\Sigma_r(\mathcal M)
 &=\{n-\wt(D):D\le C_{\mathcal M},\ \dim D=r\}.
\end{aligned}\]
This is the standard higher-support form of the projective-system correspondence.
\end{proposition}

The preceding correspondence applies to nondegenerate codes. For a general \([n,k]_q\) linear code \(C\), let \(z\) be the number of coordinate positions that vanish identically on \(C\), and puncture these positions. The resulting nondegenerate code \(C^\circ\) has length \(n_\circ=n-z\), dimension \(k\), and the same support weight on every corresponding subcode. Conversely, a code associated with a projective system of size \(n_\circ\le n\) can be extended to length \(n\) by adjoining \(n-n_\circ\) identically zero coordinates.  Therefore
\[L_r(n,k,q)=\max\left\{\left|\{\wt(D):D\le C,\ \dim D=r\}\right|:
C\le\F_q^n,\ \dim C=k\right\}.\]
Thus the total-multiplicity budget in the geometric formulation is exactly equivalent to fixing the length in the coding-theoretic formulation. For the nondegenerate code \(C_{\mathcal M}\), the quantity \(d_r(\mathcal M)\) is its \(r\)th generalized Hamming weight, and the \(r\)-MCS and \(r\)-FCS conditions correspond to the maximum and full support-weight spectrum conditions, respectively. Adjoining identically zero coordinates preserves every subcode support weight and the maximum spectrum property. It does not preserve the full spectrum property relative to the enlarged length, since the new largest weight cannot occur.
The following result is the projective reformulation of the ordinary full weight spectrum theorem of Alderson \cite{A2019}.

\begin{theorem}\cite{A2019}\label{thm:ordinary-FWS}
Let \(k\le n\). A \(1\)-FCS projective \([n,k]\) system exists if and only if \(n<2^k\).
\end{theorem}

\subsection{Separating sets}\label{subsec:cutting-blocking}
Recall Definition~\ref{def:separating-labeling}. For \(r=1\), the condition that all hyperplanes have distinct traces on a point set is the point-query version of dual non-adaptive search discussed by H\'eger, Patk\'os, and Tak\'ats \cite[Section~4]{HPT2015}. For arbitrary codimension, the following observation ensures that a separating support satisfies the rank requirement for a projective system.

\begin{proposition}\label{prop:separating-spans}
Let \(1\le r\le k-1\). Every \(r\)-separating set in \(\Pi\) spans \(\Pi\).
\end{proposition}

\begin{proof}
Put \(s=k-r\), and suppose that an \(r\)-separating set \(S\) is contained in \(\PG(W)\) for some vector hyperplane \(W<V\). Choose \(I\le W\) with \(\dim I=s-1\), a vector \(u\notin W\), and a vector \(v\in W\setminus I\). These choices are possible because \(1\le s\le k-1\). The subspaces \(U_1=I+\langle u\rangle\) and \(U_2=I+\langle u+v\rangle\) have dimension \(s\) and satisfy \(U_1\cap W=U_2\cap W=I\). They are distinct, since equality would imply \(v\in U_1\cap W=I\). Consequently, \(S\cap\PG(U_1)=S\cap\PG(I)=S\cap\PG(U_2),\) contradicting the separating property. Here \(\PG(I)=\varnothing\) is allowed when \(s=1\). Thus \(S\) spans \(\Pi\).
\end{proof}

\section{The size-constrained extremal function}\label{sec:fixed-length}
We study \(L_r(n,k,q)\) under the total-multiplicity budget \(n\). After recording its basic properties, we prove the two parts of Theorem~\ref{thm:intro-fixed-size} in sequence. Then we study the attainment of the interval bound \(n-r+1\), or equivalently the existence of \(r\)-FCS systems, in Subsection~\ref{subsec:full-spectrum}.
In Subsection~\ref{subsec:codim-one}, we determine the extremal function exactly in the point case \(r=k-1\).
Finally, in Subsection~\ref{subsec:global-local}, we establish the moment identities and local incidence bounds that restrict general spectra and supply the lower-bound tools for the \(r\)-MCS threshold in Section~\ref{sec:mcs-threshold}.

\begin{proposition}\label{prop:fixed-framework}
Let \(1\le r\le k\le n\). Then \(L_r(n,k,q)\) is nondecreasing in \(n\), and
\[k-r+1\le L_r(n,k,q)\le\min\{M_r,n-r+1\}.\]
In particular, \(L_r(k,k,q)=k-r+1\).
\end{proposition}

\begin{proof}
The monotonicity is immediate from the budget formulation. Every projective system admissible for the budget \(n\) is also admissible for the budget \(n+1\). Moreover, the upper bound is precisely the universal estimate \eqref{eq:intro-two-bounds}.

To establish the lower bound, let \(e_1,\ldots,e_k\) be a basis of \(V=\F_q^k\), and consider the simple projective system \(\mathcal B=\{[e_1],\ldots,[e_k]\}.\)
Since these points are projectively independent, \(\mathcal B\) spans \(\Pi\) and has size \(k\).
We shall show that its complementary character spectrum is precisely \(\{r,r+1,\ldots,k\}\).
Fix an integer \(t\) with \(r\le t\le k\). Let \(f_1,\ldots,f_r\) be a basis of \(\F_q^r\), and define a linear map \(\pi_t:V\to\F_q^r\) by
\[\pi_t(e_i)=
\begin{cases}
f_i, & 1\le i\le r,\\
f_1, & r<i\le t,\\
0,   & t<i\le k.
\end{cases}\]
This map is well defined because its values are specified on a basis of \(V\). Moreover, it is surjective, since the images of \(e_1,\ldots,e_r\) span \(\F_q^r\). Hence its kernel \(U_t\) has codimension \(r\), and \(\PG(U_t)\) belongs to \(\mathcal F_r(\Pi)\).
By construction, \(e_i\in U_t\) if and only if \(i>t\). Thus exactly \(k-t\) points of \(\mathcal B\) lie in \(\PG(U_t)\), while the remaining \(t\) points lie outside it. Since every point of \(\mathcal B\) has multiplicity one, the corresponding complementary character is
\(|\mathcal B|-\Char_{\mathcal B}(\PG(U_t))=k-(k-t)=t.\)
As \(t\) ranges over \(\{r,\ldots,k\}\), every integer in this interval occurs in \(\Omega_r(\mathcal B)\). Conversely, the general bound on complementary characters gives \(\Omega_r(\mathcal B)\subseteq\{r,\ldots,k\}\), because \(\mathcal B\) is spanning and has size \(k\). Therefore, \(\Omega_r(\mathcal B)=\{r,r+1,\ldots,k\}\) and \(\ell_r(\mathcal B)=k-r+1.\)
Since \(|\mathcal B|=k\le n\), it is admissible under every budget \(n\ge k\). Consequently, \(L_r(n,k,q)\ge k-r+1\), as required.
Finally, when \(n=k\), the universal interval bound gives \(L_r(k,k,q)\le k-r+1\). Together with the lower bound just established, this yields \(L_r(k,k,q)=k-r+1\).
\end{proof}

When \(r=k\), the only member of \(\mathcal F_k(\Pi)\) is the empty flat, so \(\Omega_k(\mathcal M)=\{|\mathcal M|\}\). Consequently, \(L_k(n,k,q)=1\) for every \(n\ge k\), \(\nu_k(k,q)=k\), and a \(k\)-FCS system has size exactly \(k\).

\subsection{\texorpdfstring{The existence of \(r\)-FCS systems}{The existence of r-FCS systems}}\label{subsec:full-spectrum}

We begin with the interval extremum. The following lifting construction embeds the complementary hyperplane spectrum of a projective system into a higher-order spectrum by adjoining \(r-1\) projectively independent points. The construction also yields a comparison between the corresponding size-constrained extremal functions.
For an integer \(a\) and a set \(A\subseteq\mathbb Z\), write 
\(a+A=\{\,a+x:x\in A\,\}. \)

\begin{proposition}\label{prop:lifting}
Let \(1\le r\le k\), put \(k_0=k-r+1\), and let \(\mathcal N\) be a projective \([N,k_0]\) system. Then there exists a projective \([N+r-1,k]\) system \(\mathcal M\) such that
\((r-1)+\Omega_1(\mathcal N)\subseteq\Omega_r(\mathcal M).\)
Consequently, for every \(n\ge k\),
\[L_r(n,k,q)\ge L_1(n-r+1,k-r+1,q).\]
\end{proposition}

\begin{proof}
Let \(V=E\oplus V_0\), where \(\dim E=r-1\) and \(\dim V_0=k_0=k-r+1\). Choose a basis \(e_1,\ldots,e_{r-1}\) of \(E\), and let \(h_1,\ldots,h_N\in V_0\) be nonzero vector representatives of the point occurrences of \(\mathcal N\), repeated according to their multiplicities. Consider the projective multiset \(\mathcal M\) in \(\PG(V)\) represented by
\(e_1,\ldots,e_{r-1},h_1,\ldots,h_N\), with both summands viewed as subspaces of \(V\). When \(r=1\), the first list is empty.
Since \(e_1,\ldots,e_{r-1}\) span \(E\) and the representatives of \(\mathcal N\) span \(V_0\), these vectors together span \(V\). Thus \(\mathcal M\) is a projective system of rank \(k\) and size \(N+r-1\).
We next show that the complementary hyperplane spectrum of \(\mathcal N\) is preserved, up to a translation by \(r-1\), among the complementary characters of \(\mathcal M\) on codimension-\(r\) flats. Let \(K\) be an arbitrary hyperplane of \(V_0\). Since \(\dim K=k_0-1=k-r\), its projectivization \(\PG(K)\), regarded as a subspace of \(\PG(V)\), has codimension \(r\). Here \(K=\{0\}\) is allowed when \(k_0=1\), in accordance with the convention \(\PG(\{0\})=\varnothing\).
Because \(E\cap V_0=\{0\}\), none of the \(r-1\) additional points represented by \(e_i\) belongs to \(\PG(K)\). Moreover, the remaining \(N\) occurrences are supported in \(\PG(V_0)\), and their incidences with \(\PG(K)\) are exactly those of \(\mathcal N\). Consequently,
\(\Char_{\mathcal M}(\PG(K))=\Char_{\mathcal N}(\PG(K)).\)
Since \(|\mathcal M|=N+r-1\), it follows that
\(|\mathcal M|-\Char_{\mathcal M}(\PG(K))=(r-1)+\bigl(N-\Char_{\mathcal N}(\PG(K))\bigr).\)
As \(K\) ranges over all hyperplanes of \(V_0\), each value in \(\Omega_1(\mathcal N)\)
therefore gives a value in \(\Omega_r(\mathcal M)\) obtained by adding \(r-1\). This proves
\((r-1)+\Omega_1(\mathcal N)\subseteq\Omega_r(\mathcal M).\)

It remains to establish the inequality between the extremal functions. Fix \(n\ge k\), and choose a projective system \(\mathcal N\) of rank \(k_0\) attaining \(L_1(n-r+1,k_0,q)\). By the size-budget condition, its total multiplicity satisfies \(N\le n-r+1\). The construction above then produces a projective system \(\mathcal M\) of rank \(k\) and size \(N+r-1\le n\), so \(\mathcal M\) is admissible for \(L_r(n,k,q)\).
Finally, translating a set of integers does not change its cardinality. Therefore, the spectral inclusion established above gives
\(\ell_r(\mathcal M)\ge\ell_1(\mathcal N)=L_1(n-r+1,k_0,q).\)
Since \(L_r(n,k,q)\) is the maximum of \(\ell_r\) over all projective systems admissible under the budget \(n\), we obtain \(L_r(n,k,q)\ge L_1(n-r+1,k-r+1,q),\) as required.
\end{proof}

\begin{proof}[Proof of Theorem~\ref{thm:intro-fixed-size}\textup{(i)}]
Put \(n_0=n-r+1\) and \(k_0=k-r+1\). Since \(n\ge k\), we have \(n_0\ge k_0\), while the hypothesis gives \(n_0<2^{k_0}\). By Theorem~\ref{thm:ordinary-FWS}, there exists a \(1\)-FCS projective \([n_0,k_0]\) system \(\mathcal N\). In particular, \(\Omega_1(\mathcal N)=\{1,2,\ldots,n_0\}\).
Applying Proposition~\ref{prop:lifting} to \(\mathcal N\), we obtain a projective system \(\mathcal M\) of rank \(k\) and total multiplicity \(n_0+r-1=n\), whose complementary character spectrum contains
\((r-1)+\Omega_1(\mathcal N)=\{r,r+1,\ldots,n\}.\)
On the other hand, since \(\mathcal M\) spans \(\Pi\), the occurrences outside any codimension-\(r\) flat must span the corresponding \(r\)-dimensional quotient space. Thus every complementary character lies between \(r\) and \(n\), and hence
\(\{r,r+1,\ldots,n\}\subseteq\Omega_r(\mathcal M)\subseteq\{r,r+1,\ldots,n\}.\)
Consequently,
\(\Omega_r(\mathcal M)=\{r,r+1,\ldots,n\}\),
so \(\mathcal M\) is an \(r\)-FCS projective system with \(\ell_r(\mathcal M)=n-r+1\).
Since \(\mathcal M\) is admissible under the size budget \(n\), it follows that \(L_r(n,k,q)\ge n-r+1\). The reverse inequality is the universal interval bound in \eqref{eq:intro-two-bounds}. Therefore, \(L_r(n,k,q)=n-r+1\),
as claimed.
\end{proof}

The proof of Theorem~\ref{thm:intro-fixed-size}\textup{(ii)} is provided in Subsection~\ref{subsec:codim-one}.
The sufficient condition in Theorem~\ref{thm:intro-fixed-size}\textup{(i)} is expected to be sharp. This leads to the following conjecture.

\begin{conjecture}\label{conj:rFCS}
Let \(1\le r\le k\le n\). An \(r\)-FCS projective \([n,k]\) system exists if and only if
\(n-r+1<2^{k-r+1}.\)
Equivalently,
\[L_r(n,k,q)=n-r+1\quad\Longleftrightarrow\quad n-r+1<2^{k-r+1}.\]
\end{conjecture}

The direct-sum construction has a distinctive geometric feature. There is a projective \((k-r)\)-flat containing the entire lower-dimensional projective system, while the remaining \(r-1\) point occurrences lie outside this flat. Moreover, all codimension-\(r\) flats used to realize the full spectrum are hyperplanes of this common flat. The next proposition shows that, whenever an \(r\)-FCS system admits such a common-flat structure, the numerical condition in Conjecture~\ref{conj:rFCS} is necessarily satisfied.

\begin{proposition}\label{prop:common-flat}
Let \(1\le r\le k-1\), and let \(\mathcal M\) be a projective \([n,k]\) system. Suppose that there exists a projective \((k-r)\)-flat \(\Lambda\subseteq\Pi\) such that
\(\sum_{P\notin\Lambda}m(P)=r-1.\)
Assume further that, for every \(t\in\{r,\ldots,n\}\), there is a hyperplane \(\Gamma_t\) of \(\Lambda\) satisfying
\(n-\Char_{\mathcal M}(\Gamma_t)=t.\)
Then
\[n-r+1<2^{k-r+1}.\]
\end{proposition}

\begin{proof}
Let \(W_\Lambda\le V\) be the vector subspace
corresponding to \(\Lambda\), so that
\(\dim W_\Lambda=k-r+1\). Let
\(\mathcal M|_\Lambda\) denote the restriction
of \(\mathcal M\) to \(\Lambda\), with the
multiplicities inherited from \(\mathcal M\).
Since the total multiplicity outside \(\Lambda\)
is \(r-1\), the restricted system has total
multiplicity \(n-r+1\).

We first show that \(\mathcal M|_\Lambda\) spans \(\Lambda\). Let \(R\le W_\Lambda\) be the vector span of its point occurrences. Since \(\mathcal M\) spans \(V\) and has only \(r-1\) occurrences outside \(\Lambda\), these additional occurrences can increase the dimension of \(R\) by at most \(r-1\). Consequently, \(k\le\dim R+r-1.\)
On the other hand, \(R\le W_\Lambda\) and \(\dim W_\Lambda=k-r+1\). Hence \(\dim R=k-r+1\), which implies \(R=W_\Lambda\). Thus \(\mathcal M|_\Lambda\) is a projective \([n-r+1,k-r+1]\) system.

We next show that the restricted system is \(1\)-FCS. For each \(t\in\{r,\ldots,n\}\), the hypothesis provides a hyperplane \(\Gamma_t\) of \(\Lambda\) such that
\(n-\Char_{\mathcal M}(\Gamma_t)=t.\)
Since \(\Lambda\) has projective dimension \(k-r\), each \(\Gamma_t\) is a codimension-\(r\) flat of \(\Pi\).
Moreover, \(\Gamma_t\subseteq\Lambda\), so the restriction preserves its character:
\(\Char_{\mathcal M|_\Lambda}(\Gamma_t)=\Char_{\mathcal M}(\Gamma_t).\)
It follows that
\[\begin{aligned}
(n-r+1)-\Char_{\mathcal M|_\Lambda}(\Gamma_t)
&=(n-r+1)-\Char_{\mathcal M}(\Gamma_t)\\
&=\bigl(n-\Char_{\mathcal M}(\Gamma_t)\bigr)
  -(r-1)\\
&=t-r+1.
\end{aligned}
\]
As \(t\) ranges from \(r\) to \(n\), these complementary hyperplane characters attain every integer from \(1\) to \(n-r+1\). Since \(\mathcal M|_\Lambda\) is spanning and has total multiplicity \(n-r+1\), all its complementary hyperplane characters lie in this interval. Therefore,
\(\Omega_1(\mathcal M|_\Lambda)=\{1,2,\ldots,n-r+1\}.\)
Thus \(\mathcal M|_\Lambda\) is a \(1\)-FCS projective system of rank \(k-r+1\) and size \(n-r+1\).
Applying the necessity part of Theorem~\ref{thm:ordinary-FWS}, we obtain
\(n-r+1<2^{k-r+1},\)
as required.
\end{proof}

The direct-sum lifting in Proposition~\ref{prop:lifting} is a special case of the geometry described above: the projectivization of \(V_0\) is the common \((k-r)\)-flat, while the \(r-1\) added basis points account for the entire multiplicity outside it. Thus any counterexample to the necessity direction of Conjecture~\ref{conj:rFCS} must depart from this common-flat configuration. In particular, its full spectrum cannot be realized entirely by hyperplanes of a single \((k-r)\)-flat carrying all but \(r-1\) point occurrences.

\begin{remark}\label{rem:k4r2}
Conjecture~\ref{conj:rFCS} holds for \((k,r)=(4,2)\), where the codimension-\(2\) flats are the lines of \(\PG(3,q)\) and the conjectured range is \(n\le8\). Existence for \(4\le n\le8\) is Theorem~\ref{thm:intro-fixed-size}(i). Conversely, let \(\mathcal M\) be a \(2\)-FCS projective \([n,4]\) system with \(n\ge9\), so that every integer in \(\{0,1,\ldots,n-2\}\) is the character of some line. Let \(\Lambda\) be a line with \(\Char_{\mathcal M}(\Lambda)=n-2\). Exactly two occurrences lie outside \(\Lambda\). Since \(\mathcal M\) spans \(\Pi\), they are at two distinct points \(R_1,R_2\) with \(\langle\Lambda,R_1,R_2\rangle=\Pi\). In particular, the line \(R_1R_2\) is disjoint from \(\Lambda\). Let \(\Lambda'\) be a line with \(\Char_{\mathcal M}(\Lambda')=n-3\). Then \(\Lambda'\ne\Lambda\), and \(\Lambda'\) contains at most one of \(R_1,R_2\), because \(\Char_{\mathcal M}(R_1R_2)=2<n-3\). Since \(\Lambda'\) meets \(\Lambda\) in at most one point, it follows that \(\Lambda'\cap\Lambda=\{X\}\) for a point \(X\) with \(m(X)\ge n-4\), and the occurrences on \(\Lambda\) other than those at \(X\) number \((n-2)-m(X)\le2\). Every line through \(X\) has character at least \(m(X)\ge n-4\). A line \(L\) not through \(X\) meets \(\Lambda\) in at most one point different from \(X\), which carries at most \(2\) occurrences, and \(L\) contains at most one of \(R_1,R_2\) unless \(L=R_1R_2\). Hence \(\Char_{\mathcal M}(L)\le3\). Since \(n\ge9\), the value \(n-5\) satisfies \(3<n-5<n-4\), so no line has character \(n-5\), a contradiction. The argument does not use the common-flat hypothesis of Proposition~\ref{prop:common-flat}.
\end{remark}

\subsection{\texorpdfstring{The point case \(r=k-1\)}{The point case r=k-1}}\label{subsec:codim-one}
Throughout this subsection, assume \(k\ge2\). We now prove part~\textup{(ii)} of Theorem~\ref{thm:intro-fixed-size} by treating \(r=k-1\) exactly. In this case the members of \(\mathcal F_{k-1}(\Pi)\) are precisely the projective points. Hence the character spectrum is simply the set of values assumed by the point-multiplicity function. Therefore, the incidence geometry enters only through the requirement that the positive support span \(\Pi\). This reduces the extremal problem to a discrete optimization with two independent costs: an arithmetic cost for producing many distinct nonnegative multiplicities and a geometric cost for ensuring that the points carrying positive multiplicity span the ambient projective space.

Extend every multiplicity function by zero to all points of \(\Pi\), and, for \(2\le m\le\vartheta_k\), put
\[\mu_k(m)=\binom{m}{2}+\max\{0,k-m+1\}.\]
The next lemma identifies the minimum total multiplicity required to realize \(m\) distinct point characters.

\begin{lemma}\label{lem:codim-cost}
Let \(2\le m\le\vartheta_k\), and let \(a:\Pi\to\Z_{\ge0}\) have spanning positive support and exactly \(m\) distinct values. Then
\[\sum_{P\in\Pi}a(P)\ge\mu_k(m).\]
Conversely, for every \(2\le m\le\vartheta_k\), equality is attained by such a function \(a\).
\end{lemma}

\begin{proof}
Let \(p\) be the number of distinct positive values assumed by \(a\), and write them as
\(1\le u_1<u_2<\cdots<u_p\).
Since these are distinct positive integers, \(u_i\ge i\) for every \(i\). Thus choosing one point carrying each positive value already contributes at least
\(1+2+\cdots+p=\binom{p+1}{2}\) to the total multiplicity.
There is a second, independent contribution coming from spanning. The positive support of \(a\) spans \(\PG(k-1,q)\), and therefore contains at least \(k\) projectively independent points. If the \(p\) points chosen above do not yet provide \(k\) support points, at least \(k-p\) additional points of positive multiplicity are required, each contributing at least one. Consequently,
\[\sum_{P\in\Pi}a(P)\ge\binom{p+1}{2}+\max\{0,k-p\}.\]

We now relate \(p\) to the total number \(m\) of distinct values. If zero occurs among the values of \(a\), then \(p=m-1\), and the preceding bound becomes exactly
\(\binom{m}{2}+\max\{0,k-m+1\}=\mu_k(m).\)
If zero does not occur, then \(p=m\). In that case the same lower bound is strictly larger than \(\mu_k(m)\): the excess is \(m-1\) when \(m\le k\), and \(m\) when \(m\ge k+1\). In particular, equality in the asserted bound is possible only when zero is one of the character values.

It remains to show that the bound is sharp. First suppose \(2\le m\le k\). Choose \(k\) projectively independent points \(P_1,\ldots,P_k\). Assign multiplicities \(1,2,\ldots,m-1\) to \(P_1,\ldots,P_{m-1}\), respectively, assign multiplicity \(1\) to the remaining \(k-m+1\) points, and assign zero to every other point of \(\Pi\). The positive support spans \(\Pi\). Since \(k<\vartheta_k\), at least one point has value zero, so the value set is exactly \(\{0,1,\ldots,m-1\}\). The total multiplicity is
\(\binom{m}{2}+k-m+1=\mu_k(m).\)
Now suppose \(k+1\le m\le\vartheta_k\). Choose \(m-1\) projective points containing \(k\) projectively independent points, and assign the values \(1,2,\ldots,m-1\) bijectively to them. Assign zero to all remaining points. At least one such point exists because \(m-1\le\vartheta_k-1\). The positive support spans \(\Pi\), the value set has exactly \(m\) elements, and the total multiplicity is \(\binom{m}{2}=\mu_k(m)\). This proves the sharpness in all cases.
\end{proof}

\begin{proof}[Proof of Theorem~\ref{thm:intro-fixed-size}\textup{(ii)}]
Let \(\mathcal M\) be any projective system admissible for the budget \(n\), and write \(N=|\mathcal M|\le n\). Extend its multiplicity function to
\(a:\Pi\longrightarrow\Z_{\ge0}\)
by setting \(a(P)=0\) outside the support of \(\mathcal M\).
Because \(r=k-1\), every member of \(\mathcal F_{k-1}(\Pi)\) is a single projective point. Thus the character spectrum is exactly \(a(\Pi)\), while the complementary spectrum is obtained from it by the affine involution \(x\mapsto N-x\). This map is bijective and therefore preserves the number of distinct values, and hence \(\ell_{k-1}(\mathcal M)=|a(\Pi)|\).
Let \(m=\ell_{k-1}(\mathcal M)\). Since \(k\ge2\) and \(n\ge k\), the value \(m=2\) is always feasible. Indeed, \(\mu_k(2)=k\le n\). Hence a system with only one character value cannot be extremal, and we may assume \(m\ge2\).
The positive support of \(a\) spans \(\Pi\), so Lemma~\ref{lem:codim-cost} gives \(\mu_k(m)\le N\le n\). Moreover, \(a\) is defined on the \(\vartheta_k\) points of \(\Pi\), so it can assume at most \(\vartheta_k\) distinct values. Therefore every attainable spectrum size \(m\ge2\) satisfies \(m\le\vartheta_k\) and \(\mu_k(m)\le n\); the case \(m=1\) cannot increase the maximum.
This proves the upper bound in the formula of Theorem~\ref{thm:intro-fixed-size}\textup{(ii)}.

Conversely, suppose that an integer \(m\) satisfies these two conditions. By the sharpness part of Lemma~\ref{lem:codim-cost}, there exists a multiplicity function with spanning positive support, exactly \(m\) distinct values, and total multiplicity \(\mu_k(m)\le n\).  Therefore, it defines a projective system admissible for the budget \(n\) with \(\ell_{k-1}=m\). Maximizing over all such \(m\) proves the exact formula for \(L_{k-1}(n,k,q)\).

We next determine the \(r\)-MCS threshold in this point case. Since \(M_{k-1}=\qbinom{k}{k-1}=\vartheta_k,\) a \((k-1)\)-MCS system must have a different character at every projective point. Thus its extended multiplicity function assumes exactly \(\vartheta_k\) distinct values. Applying the multiplicity-cost lemma with \(m=\vartheta_k\), and using \(\vartheta_k\ge k+1\), shows that its total multiplicity is at least
\(\mu_k(\vartheta_k)=\binom{\vartheta_k}{2}\).
The equality construction in Lemma~\ref{lem:codim-cost} attains this value, and hence
\(\nu_{k-1}(k,q)=\binom{\vartheta_k}{2}\).

Finally, consider the \(r\)-FCS equality case. Let \(\mathcal M\) be a \((k-1)\)-FCS projective system of total multiplicity \(n\). Its complementary character spectrum is \(\{k-1,k,\ldots,n\}\), and therefore it contains \(m=n-k+2\) distinct values. Since \(n=k+m-2\), the multiplicity-cost lemma requires \(\mu_k(m)\le k+m-2\).
If \(m\le k\), then \(\mu_k(m)=k+\binom{m-1}{2}\), so the difference between the two sides is
\(\mu_k(m)-(k+m-2)=\frac{(m-2)(m-3)}{2}.\)
If \(m\ge k+1\), then \(\mu_k(m)=\binom m2\), and the difference is \((m-2)(m-3)/2+(m-k-1)\), which is no smaller. It follows that the inequality can hold only for \(m\le3\). Since \(m=n-k+2\), necessarily \(n\in\{k,k+1\}\).
Conversely, both values occur. Indeed, when \(r=k-1\), the hypothesis of Theorem~\ref{thm:intro-fixed-size}\textup{(i)} becomes \(n-k+2<4\), which holds exactly for \(n=k\) and \(n=k+1\). This proves all the assertions in part~\textup{(ii)}.
\end{proof}

The proof of Lemma~\ref{lem:codim-cost} also determines the structure of every extremal multiplicity profile.

\begin{proposition}\label{prop:codim-equality}
Let \(2\le m\le\vartheta_k\), and let \(a:\Pi\to\Z_{\ge0}\) have spanning positive support, exactly \(m\) distinct values, and satisfy \(\sum_{P\in\Pi}a(P)=\mu_k(m).\) Then zero is one of the values of \(a\), and exactly one of the following holds.
\begin{enumerate}
\item[(i)] If \(2\le m\le k\), the positive support consists of exactly \(k\) projectively independent points. The value \(1\) occurs \(k-m+2\) times, while each of \(2,\ldots,m-1\) occurs exactly once.

\item[(ii)] If \(k+1\le m\le\vartheta_k\), the positive support consists of exactly \(m-1\) points containing \(k\) projectively independent points, and the values \(1,\ldots,m-1\) each occur exactly once.
\end{enumerate}
\end{proposition}

\begin{proof}
The lower bound in Lemma~\ref{lem:codim-cost} is obtained by combining two independent inequalities. The first is arithmetic: \(p\) distinct positive integers have total at least \(1+\cdots+p\). The second is geometric: a spanning positive support contains at least \(k\) points. Equality in the final bound therefore forces equality in both steps.
In particular, zero must occur among the \(m\) values, so there are exactly \(m-1\) distinct positive values. Equality in the arithmetic bound forces these values to be precisely \(1,2,\ldots,m-1\), and each value greater than one must occur at exactly one point.

Suppose first that \(2\le m\le k\). The \(m-1\) points realizing the distinct positive values do not provide enough points to meet the minimum spanning cardinality. Exactly \(k-m+1\) additional positive support points are therefore required. Equality in the total multiplicity forces each additional point to have multiplicity one. Hence the support has exactly \(k\) points; because it spans \(\Pi\), these points are projectively independent. The value \(1\) occurs once among the representatives of the distinct positive values and \(k-m+1\) further times, for a total of \(k-m+2\).

If \(m\ge k+1\), the \(m-1\) representatives of the distinct positive values are already numerous enough to span. Equality leaves no multiplicity available for any additional positive support point. Thus the support consists of exactly these \(m-1\) points, each carrying one of the values \(1,\ldots,m-1\), and the spanning condition guarantees that they contain \(k\) projectively independent points.
\end{proof}

The exact formula in Theorem~\ref{thm:intro-fixed-size}\textup{(ii)} can also be written in closed form. Since
\[\mu_k(m)=\begin{cases}
k+\binom{m-1}{2}, & 2\le m\le k,\\[1mm]
\binom m2, & m\ge k+1,
\end{cases}\]
solving \(\mu_k(m)\le n\) gives
\begin{equation}\label{eq:codim-closed-form}
L_{k-1}(n,k,q)=\begin{cases}
\displaystyle
1+\left\lfloor\dfrac{1+\sqrt{1+8(n-k)}}{2}\right\rfloor,
&k\le n<\binom{k+1}{2},
\\[4mm]
\displaystyle
\min\left\{\vartheta_k,\,\left\lfloor\dfrac{1+\sqrt{1+8n}}{2}\right\rfloor\right\},
&n\ge\binom{k+1}{2}.
\end{cases}
\end{equation}

The equality classification also determines the multiplicity profiles of minimum-size \((k-1)\)-MCS systems. It does not assert that all systems with these profiles are projectively equivalent.

\begin{corollary}\label{cor:minimum-point-MCS}
Let \(\mathcal M\) be a minimum-size \((k-1)\)-MCS projective system. After extending its multiplicity function by zero to all points of \(\Pi\), its values are
\(0,1,\ldots,\vartheta_k-1\), each occurring exactly once.
\end{corollary}

\begin{proof}
By Theorem~\ref{thm:intro-fixed-size}\textup{(ii)}, a minimum-size \((k-1)\)-MCS system attains the multiplicity-cost bound with \(m=\vartheta_k\). Since \(\vartheta_k\ge k+1\), Proposition~\ref{prop:codim-equality}\textup{(ii)} applies and gives the asserted profile.
\end{proof}

As a simple consequence of the closed formula, for \(k\ge3\), \(L_{k-1}(k+2,k,q)=3\),
whereas the interval upper bound equals \(4\). Thus the elementary size-budget bound already fails two units beyond the minimum budget. Moreover, the characterization of \((k-1)\)-FCS systems obtained above verifies Conjecture~\ref{conj:rFCS} completely in the point case.

\subsection{Global moment identities and local character obstructions}
\label{subsec:global-local}

The elementary bounds developed above depend only on the cardinality of the Grassmannian and on the range of possible character values. We now exploit the incidence structure of the Grassmannian together with the rank structure of the projective system. There are two complementary viewpoints. 
Globally, binomial moments of the character distribution can be evaluated by double counting labeled occurrences. Locally, the largest character carried by a flat measures how strongly the multiset can concentrate inside a subspace. The interaction between these two viewpoints gives substantially stronger restrictions on large character spectra.

We begin with the global moment identity stated in Theorem~\ref{thm:intro-moments}.

\begin{proof}[Proof of Theorem~\ref{thm:intro-moments}]
Let \(\mathcal M\) have total multiplicity \(N\), and label its point occurrences by \(P_1,\ldots,P_N\), distinguishing occurrences with different indices even when they represent the same projective point. Fix an integer \(t\ge0\), and consider the incidence set
\[\mathcal X_t=\left\{(\Lambda,J):\Lambda\in\mathcal F_r(\Pi),\ J\subseteq[N],\ |J|=t,\ P_j\in\Lambda\text{ for all }j\in J\right\}.\]
We count \(\mathcal X_t\) in two ways.
For a fixed \(\Lambda\in\mathcal F_r(\Pi)\), exactly \(\Char_{\mathcal M}(\Lambda)\) labeled occurrences lie in \(\Lambda\). Hence there are \(\binom{\Char_{\mathcal M}(\Lambda)}{t}\) possible choices of \(J\). Summing over all codimension-\(r\) flats, we have
\(|\mathcal X_t|=\sum_{\Lambda\in\mathcal F_r(\Pi)}\binom{\Char_{\mathcal M}(\Lambda)}{t}.\)

Conversely, fix \(J\subseteq[N]\) with \(|J|=t\), and let \(R_J\le V\) be the vector span of representatives of the occurrences indexed by \(J\).
By definition, \(\dim R_J=\rho_{\mathcal M}(J)\). A codimension-\(r\) flat \(\PG(U)\) contains all these occurrences if and only if \(R_J\subseteq U\).
The correspondence \(U\mapsto U/R_J\) gives a bijection between codimension-\(r\) subspaces of \(V\) containing \(R_J\) and codimension-\(r\) subspaces of \(V/R_J\). Since \(\dim(V/R_J)=k-\rho_{\mathcal M}(J)\), their number is
\(\qbinom{k-\rho_{\mathcal M}(J)}{r}.\)
If \(\rho_{\mathcal M}(J)>k-r\), no such subspace exists, and the Gaussian coefficient is zero by convention. Summing over all \(t\)-subsets \(J\), we have
\(|\mathcal X_t|=\sum_{\substack{J\subseteq[N]\\|J|=t}}\qbinom{k-\rho_{\mathcal M}(J)}{r}.\)
Comparing the two expressions for \(|\mathcal X_t|\), we prove \eqref{eq:binomial-moments}. This also covers \(t=0\), when \(J=\varnothing\) and \(R_J=\{0\}\), and \(t>N\), when both sums vanish.

We next prove \eqref{eq:moment-image-bound}. Assume \(1\le r\le k-1\) and \(t\ge1\), and write the distinct character values in increasing order as \(0\le b_0<b_1<\cdots<b_{\ell-1}\), where \(\ell=\ell_r(\mathcal M)\).
Choose one flat realizing each \(b_i\). Since the \(b_i\) are distinct nonnegative integers, we have \(b_i\ge i\). Moreover, \(x\mapsto\binom{x}{t}\) is nondecreasing on the nonnegative integers. Therefore, the selected flats contribute at least \(\sum_{i=0}^{\ell-1}\binom{i}{t}\) to the \(t\)th binomial moment. As all remaining contributions are nonnegative, the hockey-stick identity yields
\[\sum_{\Lambda\in\mathcal F_r(\Pi)}\binom{\Char_{\mathcal M}(\Lambda)}{t}
\ge\sum_{i=0}^{\ell-1}\binom{b_i}{t}\ge\sum_{i=0}^{\ell-1}\binom{i}{t}
=\binom{\ell}{t+1}.\]

For the upper bound, every \(t\)-subset \(J\subseteq[N]\) is nonempty, so \(\rho_{\mathcal M}(J)\ge1\). Consequently,
\[\qbinom{k-\rho_{\mathcal M}(J)}{r}\le\qbinom{k-1}{r}=B_r.\]
Since there are \(\binom{N}{t}\) choices of \(J\), the moment identity gives
\[\sum_{\substack{J\subseteq[N]\\|J|=t}}\qbinom{k-\rho_{\mathcal M}(J)}{r}\le B_r\binom{N}{t}.\]
Combining the last two inequalities with \eqref{eq:binomial-moments}, we prove \eqref{eq:moment-image-bound}.
\end{proof}

The first two moments already reveal different aspects of the geometry.
The first moment depends only on the total multiplicity, whereas the second detects whether distinct occurrences are supported at the same projective point.

\begin{corollary}\label{cor:first-moment-rigidity}
Let \(\mathcal M\) be a projective \([N,k]\) system, and let \(1\le r\le k-1\). Then
\begin{equation}\label{eq:first-image-bound}
\binom{\ell_r(\mathcal M)}2\le NB_r.
\end{equation}
Moreover, equality holds if and only if the multiset of character
values on \(\mathcal F_r(\Pi)\) is
\[
\underbrace{0,\ldots,0}_{M_r-\ell_r(\mathcal M)+1\ {\rm times}},
1,2,\ldots,\ell_r(\mathcal M)-1.
\]
\end{corollary}

\begin{proof}
Setting \(t=1\) in Theorem~\ref{thm:intro-moments} gives the exact
first-moment identity
\begin{equation}\label{eq:first-moment-exact}
\sum_{\Lambda\in\mathcal F_r(\Pi)}\Char_{\mathcal M}(\Lambda)=NB_r.
\end{equation}
Geometrically, this simply counts incident pairs consisting of a labeled point occurrence of \(\mathcal M\) and a codimension-\(r\) flat containing it: every occurrence lies on exactly \(B_r\) such flats.
Let \(0\le b_0<\cdots<b_{\ell-1}\) be the distinct character values, where \(\ell=\ell_r(\mathcal M)\). Choosing one flat for each value and using \(b_i\ge i\), the contribution from these \(\ell\) flats is at least
\(0+1+\cdots+(\ell-1)=\binom{\ell}{2}\).
Every other flat has nonnegative character, and therefore \eqref{eq:first-image-bound} follows from \eqref{eq:first-moment-exact}.

Suppose equality holds. Then no loss can occur in either step of the preceding argument. Thus \(b_i=i\) for every \(i\), and every flat not chosen to represent one of the distinct values must have character zero. Consequently, each positive value \(1,\ldots,\ell-1\) occurs exactly once, while the remaining \(M_r-\ell+1\) flats all have character zero. Conversely, this character-value distribution has total first moment \(\binom{\ell}{2}\), so \eqref{eq:first-moment-exact} implies equality in \eqref{eq:first-image-bound}.
\end{proof}

Since every system admissible for the budget \(n\) has size \(N\le n\), Corollary~\ref{cor:first-moment-rigidity} yields
\begin{equation}\label{eq:explicit-first-moment-bound}
L_r(n,k,q)\le\min\left\{M_r,\,n-r+1,\,
\left\lfloor\frac{1+\sqrt{1+8nB_r}}{2}\right\rfloor\right\}.
\end{equation}
In particular, if \(\mathcal M\) is an \(r\)-FCS projective system of size \(N\), then
\begin{equation}\label{eq:complete-first-moment-necessary}
\binom{N-r+1}{2}\le NB_r.
\end{equation}

The second moment refines this information by detecting repeated projective points. Define the number of parallel pairs of occurrences by
\[
P_2(\mathcal M)
=
\sum_{P\in\Pi}\binom{m(P)}2.
\]

\begin{proposition}\label{prop:parallelism-certificate}
Let \(1\le r\le k-1\), and let \(\mathcal M\) be a projective
\([N,k]\) system. Then
\begin{equation}\label{eq:second-moment-parallel}
\sum_{\Lambda\in\mathcal F_r(\Pi)}\binom{\Char_{\mathcal M}(\Lambda)}2
=B_r^{(2)}\binom N2+\bigl(B_r-B_r^{(2)}\bigr)P_2(\mathcal M).
\end{equation}
Consequently,
\begin{equation}\label{eq:parallel-lower}
P_2(\mathcal M)\ge\max\left\{0,\,\left\lceil\frac{\binom{\ell_r(\mathcal M)}{3}-B_r^{(2)}\binom{N}{2}}{B_r-B_r^{(2)}}\right\rceil\right\}.
\end{equation}
\end{proposition}

\begin{proof}
Apply Theorem~\ref{thm:intro-moments} with \(t=2\). A pair of labeled occurrences has only two possible ranks. It has rank one exactly when the two occurrences represent the same projective point. There are \(P_2(\mathcal M)\) such pairs, and each is contained in exactly \(B_r\) codimension-\(r\) flats.
Every remaining pair consists of two distinct projective points and therefore spans a projective line. Such a line is contained in exactly \(B_r^{(2)}\) members of \(\mathcal F_r(\Pi)\). Since there are \(\binom N2-P_2(\mathcal M)\) rank-two pairs, separating the two rank strata in the moment identity gives \eqref{eq:second-moment-parallel}.

To derive the lower bound on \(P_2(\mathcal M)\), again write \(0\le b_0<\cdots<b_{\ell-1}\) for the distinct character values. One representative flat for each value contributes at least
\(\sum_{i=0}^{\ell-1}\binom{i}{2}=\binom{\ell}{3}\) 
to the second moment. Substitution into \eqref{eq:second-moment-parallel} gives a lower bound for the parallel pair count. Finally, the Gaussian recurrence gives
\[B_r-B_r^{(2)}=q^{k-r-1}\qbinom{k-2}{r-1}>0,\]
so division is legitimate. Since \(P_2(\mathcal M)\) is an integer, taking the ceiling and combining with the trivial inequality \(P_2(\mathcal M)\ge0\) proves \eqref{eq:parallel-lower}.
\end{proof}

For a simple projective system \(P_2(\mathcal M)=0\). Hence
\begin{equation}\label{eq:projective-second-moment}
\binom{\ell_r(\mathcal M)}{3}\le B_r^{(2)}\binom{N}{2}.
\end{equation}
Thus the second moment gives a quantitative form of the following phenomenon: once the character spectrum is sufficiently large, repeated projective points are unavoidable.
The moment identities are global: they control how the total character mass is distributed over the Grassmannian, but they do not locate that mass. We now introduce a local concentration parameter.

For \(0\le s\le k\), define the \emph{maximum rank-\(s\) character} of \(\mathcal M\) by
\[\alpha_s(\mathcal M)=\max\{\Char_{\mathcal M}(\PG(U)):U\in\Gr_q(s,k)\}.\]
We set \(\alpha_0(\mathcal M)=0\). This quantity is invariant under projective equivalence. When \(s=k-r\), it is precisely the maximum character attained on the codimension-\(r\) flats.

For a projective \([N,k]\) system and \(s=k-r\), where \(1\le r\le k-1\), the definitions give
\begin{equation}\label{eq:dr-flat-density}
d_r(\mathcal M)=N-\alpha_s(\mathcal M).
\end{equation}
Indeed, minimizing \(N-\Char_{\mathcal M}(\Lambda)\) is equivalent to maximizing \(\Char_{\mathcal M}(\Lambda)\). Since each character is an integer between \(0\) and \(\alpha_s(\mathcal M)\), we also have \(\Omega_r(\mathcal M)\subseteq\{N-\alpha_s(\mathcal M),\ldots,N\}\) and \(\ell_r(\mathcal M)\le\alpha_s(\mathcal M)+1.\)
The two extremal classes impose immediate conditions on this local parameter. If \(\mathcal M\) is \(r\)-FCS, then
\begin{equation}\label{eq:complete-flat-character-equality}
\alpha_{k-r}(\mathcal M)=N-r.
\end{equation}
If \(\mathcal M\) is \(r\)-MCS, then its character spectrum contains \(M_r\) distinct nonnegative integers, so necessarily
\begin{equation}\label{eq:injective-flat-character-lower}
\alpha_{k-r}(\mathcal M)\ge M_r-1.
\end{equation}

There is also a lower bound on the same quantity coming from the first moment. Since the average character over the \(M_r\) codimension-\(r\) flats is \(NB_r/M_r\), we obtain
\begin{equation}\label{eq:average-flat-density-bridge}
\alpha_{k-r}(\mathcal M)\ge\left\lceil\frac{NB_r}{M_r}\right\rceil.
\end{equation}
Together with \eqref{eq:dr-flat-density}, this gives
\(d_r(\mathcal M)\le N-\left\lceil\frac{NB_r}{M_r}\right\rceil.\)
Thus the first moment forces some concentration of the multiset inside a codimension-\(r\) flat. We next show that matroidal restrictions can, in the opposite direction, prevent such concentration.

For integers \(s\ge1\) and \(\delta\ge3\), let \(\operatorname{ex}_q(s,\delta)\) denote the maximum number of elements of a \(q\)-representable matroid of rank at most \(s\) and girth at least \(\delta\). This maximum is finite. The girth condition excludes loops and parallel pairs, so any such representation can be identified with a set of distinct points of \(\PG(s-1,q)\). In particular, \(\operatorname{ex}_q(s,\delta)\le\theta_q(s-1)\).

\begin{proposition}\label{prop:matroid-girth}
Let \(\mathcal M\) be a projective \([N,k]\) system. Let \(1\le r\le k-1\), put \(s=k-r\), and let \(\delta\ge3\) be an integer. Suppose that \(\gamma(\mathcal M)\ge\delta\). Then
\[\alpha_s(\mathcal M)\le\operatorname{ex}_q(s,\delta)~{\rm and}~
\ell_r(\mathcal M)\le\operatorname{ex}_q(s,\delta)+1.\]
\end{proposition}

\begin{proof}
Fix \(U\in\Gr_q(s,k)\), and restrict a matrix realization of \(\mathcal M\) to those labeled columns whose projective points lie in \(\PG(U)\). Their span is contained in \(U\), so the resulting \(q\)-representable matroid has rank at most \(s\).
Restriction cannot create a new circuit: every circuit among the restricted columns is already a circuit of the matroid represented by the full realization of \(\mathcal M\). Hence the restricted matroid has girth at least \(\delta\). The number of its elements, counted as labeled occurrences, is exactly \(\Char_{\mathcal M}(\PG(U))\). 
By the definition of \(\operatorname{ex}_q(s,\delta)\), we have
\(\Char_{\mathcal M}(\PG(U))\le\operatorname{ex}_q(s,\delta).\)
Since this holds for every \(U\in\Gr_q(s,k)\), maximizing over \(U\) gives the first assertion. The second follows from \(\ell_r(\mathcal M)\le\alpha_s(\mathcal M)+1\).
\end{proof}

Under the hypotheses of Proposition~\ref{prop:matroid-girth}, \eqref{eq:complete-flat-character-equality} gives
\begin{equation}\label{eq:girth-extremal-consequences}
\mathcal M\text{ is }r\text{-FCS}
\quad\Longrightarrow\quad
N\le\operatorname{ex}_q(k-r,\delta)+r.
\end{equation}

Two immediate consequences illustrate the local compression mechanism. If \(\mathcal M\) is simple, then an \(s\)-dimensional vector subspace contains at most \(\theta_q(s-1)\) points of \(\mathcal M\). Hence, with \(s=k-r\),
\begin{equation}\label{eq:projective-compression}
\ell_r(\mathcal M)\le\theta_q(k-r-1)+1.
\end{equation}
In particular, no simple projective system is \(r\)-MCS for \(1\le r\le k-1\), because
\[M_r=\qbinom{k}{s}\ge\qbinom{s+1}{s}=\theta_q(s)>\theta_q(s-1)+1.\]
Consequently, every \(r\)-MCS projective system in this range has \(\gamma(\mathcal M)=2\). Indeed, it has no loops and must contain a repeated projective point, hence a parallel pair. Thus girth assumptions of at least three restrict general spectra and the FCS class, but cannot provide additional conditions on MCS systems.

At the other extreme, if \(\gamma(\mathcal M)\ge s+2\), then no \(s\)-dimensional vector subspace can contain \(s+1\) occurrences: any such collection would be dependent and hence contain a circuit of size at most \(s+1\). Consequently,
\[\alpha_s(\mathcal M)\le s~~{\rm and}~~\ell_r(\mathcal M)\le s+1=k-r+1.\]
This applies, for instance, to projective systems whose represented matroid is uniform of rank \(k\).

The intermediate girth condition \(\gamma(\mathcal M)\ge s+1\) is naturally related to finite projective arcs. For \(s\ge2\), an arc in \(\PG(s-1,q)\) is a set of at least \(s\) points every \(s\) of which are projectively independent. Let \(a_q(s)\) be the maximum size of such an arc. The classical bound is \(a_q(s)\le q+s-1\); see \cite[Section~5]{BallLavrauw2019}.

\begin{corollary}\label{cor:arc-mds-bound}
Let \(\mathcal M\) be a projective \([N,k]\) system, let \(1\le r\le k-2\), and put \(s=k-r\). If
\(\gamma(\mathcal M)\ge s+1\), then
\[\ell_r(\mathcal M)\le a_q(s)+1\le q+s.\]
\end{corollary}

\begin{proof}
Fix \(U\in\Gr_q(s,k)\). The girth assumption excludes parallel pairs, so the occurrences supported on \(\PG(U)\) are distinct projective points. If there are fewer than \(s\) of them, their number is at most \(s-1<a_q(s)\), since \(s\) independent points form an arc.
Otherwise every \(s\) of these points are independent. A dependent \(s\)-subset would contain a circuit of size at most \(s\), contrary to \(\gamma(\mathcal M)\ge s+1\). The points therefore form an arc in \(\PG(U)\cong\PG(s-1,q)\), and again their number is at most \(a_q(s)\).
Thus \(\Char_{\mathcal M}(\PG(U))\le a_q(s)\) for every \(U\). Taking the maximum and using \(\ell_r(\mathcal M)\le\alpha_s(\mathcal M)+1\) gives \(\ell_r(\mathcal M)\le a_q(s)+1\). Then the cited classical arc bound yields \(a_q(s)+1\le q+s\).
\end{proof}

Consequently, under the hypotheses of Corollary~\ref{cor:arc-mds-bound}, an \(r\)-FCS system satisfies \(N-r+1=\ell_r(\mathcal M)\le q+k-r\), and hence \(N\le q+k-1\).

\section{The \texorpdfstring{\(r\)-MCS}{r-MCS} threshold}\label{sec:mcs-threshold}
Throughout this section, assume \(1\le r\le k-1\). We study the minimum total multiplicity \(\nu_r(k,q)\) required to distinguish every codimension-\(r\) flat. Subsection~\ref{subsec:threshold-lower} derives the first-moment lower bound, describes its endpoint defects, and tests their compatibility with the second moment. Subsection~\ref{subsec:threshold-upper} uses additive labelings to construct \(r\)-MCS systems and proves Theorem~\ref{thm:intro-threshold}. In particular, the first subsection gives necessary conditions on the character values, while the second produces actual systems with a controlled total multiplicity.

\subsection{Incidence lower bounds and endpoint defects}\label{subsec:threshold-lower}

Let \(\mathcal M\) be an \(r\)-MCS projective system. Since the character map is injective on the \(M_r\) members of \(\mathcal F_r(\Pi)\), its character values are \(M_r\) distinct nonnegative integers. Their minimum possible total is therefore \(H_r\). The following proposition refines this elementary observation by retaining the complete defect from the minimum.

\begin{proposition}\label{prop:mcs-first-moment}
Let \(\mathcal M\) be an \(r\)-MCS projective \([N,k]\) system. Write its character values in increasing order as \(0\le b_0<b_1<\cdots<b_{M_r-1}.\)
Set \(\lambda=b_0\) and \(\varepsilon_i=b_i-\lambda-i\) for \(0\le i\le M_r-1\).
Then
\(0=\varepsilon_0\le\varepsilon_1\le\cdots\le\varepsilon_{M_r-1}\),
and
\begin{equation}\label{eq:first-defect}
NB_r=H_r+M_r\lambda+\sum_{i=0}^{M_r-1}\varepsilon_i.
\end{equation}
Consequently,
\begin{equation}\label{eq:incidence-lower}
N\ge N_0:=\left\lceil\frac{H_r}{B_r}\right\rceil.
\end{equation}
If \(N=N_0\) and \(R_r=B_rN_0-H_r,\)
then
\(\lambda=0\), \(\sum_{i=0}^{M_r-1}\varepsilon_i=R_r\), and \(0\le R_r<B_r.\)
\end{proposition}

\begin{proof}
Because the \(b_i\) are strictly increasing integers, consecutive values differ by at least one. Hence \(\varepsilon_{i+1}-\varepsilon_i=b_{i+1}-b_i-1\ge0.\)
Since \(\varepsilon_0=b_0-\lambda=0\), the sequence \((\varepsilon_i)\) is nonnegative and nondecreasing. Thus \(\lambda\) measures a common translation of the character-value set, while the \(\varepsilon_i\) measure the accumulated gaps beyond the smallest possible spacing by consecutive integers.

We now use the first-moment identity \eqref{eq:first-moment-exact}. Since \(\mathcal M\) is \(r\)-MCS, each member of \(\mathcal F_r(\Pi)\) has a distinct character, and therefore each of the values \(b_0,\ldots,b_{M_r-1}\) occurs exactly once. Hence \(NB_r=\sum_{i=0}^{M_r-1}b_i.\)
Substituting \(b_i=\lambda+i+\varepsilon_i\), the contribution from the consecutive part is
\(\sum_{i=0}^{M_r-1}i=\binom{M_r}{2}=H_r.\)
This gives \eqref{eq:first-defect}.
Both defect terms on the right-hand side of \eqref{eq:first-defect} are nonnegative, so \(NB_r\ge H_r\). Since \(N\) is an integer, this proves \eqref{eq:incidence-lower}.

Suppose now that \(N=N_0\). By the definition of the ceiling, \(R_r=B_rN_0-H_r\) satisfies \(0\le R_r<B_r\), and \eqref{eq:first-defect} becomes
\(R_r=M_r\lambda+\sum_{i=0}^{M_r-1}\varepsilon_i.\)
The Gaussian coefficients satisfy
\(\frac{M_r}{B_r}=\frac{q^k-1}{q^{k-r}-1}>1,\)
and hence \(M_r>B_r>R_r\). If \(\lambda\ge1\), then already \(M_r\lambda>R_r\), which is impossible. Therefore \(\lambda=0\), and the remaining defect is entirely carried by the gaps:
\(\sum_{i=0}^{M_r-1}\varepsilon_i=R_r.\)
This proves the endpoint assertions.
\end{proof}

The first-moment lower bound always dominates the interval bound. More precisely,
\begin{equation}\label{eq:lower-bound-comparison}
\frac{H_r}{B_r}\ge M_r+r-1,\quad
\frac{H_r}{B_r}=M_r+r-1
\quad\Longleftrightarrow\quad(q,r)=(2,1).
\end{equation}
To verify this, first note that \(M_r-1\ge q+\cdots+q^r\ge2r\). Unless \((q,r)=(2,1)\), we have \(q^r\ge3\), and hence
\[\frac{M_r}{B_r}=\frac{q^k-1}{q^{k-r}-1}>q^r\ge3\ge2+\frac{2r}{M_r-1}.\]
Multiplying by \((M_r-1)/2\) gives the strict inequality in \eqref{eq:lower-bound-comparison}. If \((q,r)=(2,1)\), then \(M_1-1=2B_1\), so \(H_1/B_1=M_1\), giving equality. Thus \(N_0=\lceil H_r/B_r\rceil\) is the relevant first-moment endpoint throughout the nontrivial range. For \(r=1\), the numerical bound \(N\ge\lceil H_r/B_r\rceil\) recovers the length bound for maximum weight spectrum codes in \cite[Lemma~5.1]{AN2019}. In this case, identity \eqref{eq:first-defect} is equivalent to \cite[Theorem~12]{MK2020}: the defect \(M_1\lambda+\sum_i\varepsilon_i\) is the \emph{spread} of the corresponding maximum weight spectrum code introduced there.
The defect identity retains the translation and gap terms for arbitrary codimension.
At the first-moment endpoint, the character-value set is therefore highly constrained: its smallest value must be zero, and the entire departure from the consecutive set
\(\{0,1,\ldots,M_r-1\}\) is encoded by a nondecreasing sequence of nonnegative integers whose total is \(R_r\). The second moment provides a further arithmetic test for these possible gap profiles.

\begin{proposition}\label{prop:second-moment-obstruction}
Let \(\mathcal M\) be an \(r\)-MCS projective system of total multiplicity \(N\), with character values \(0\le b_0<b_1<\cdots<b_{M_r-1}.\)
Then
\begin{equation}\label{eq:second-obstruction-general}
\bigl(B_r-B_r^{(2)}\bigr)P_2(\mathcal M)=
\sum_{i=0}^{M_r-1}\binom{b_i}{2}-B_r^{(2)}\binom{N}{2}.
\end{equation}
In particular, the right-hand side is a nonnegative integer multiple of \(B_r-B_r^{(2)}\).
\end{proposition}

\begin{proof}
The second-moment identity \eqref{eq:second-moment-parallel} gives
\[\sum_{\Lambda\in\mathcal F_r(\Pi)}\binom{\Char_{\mathcal M}(\Lambda)}{2}
=B_r^{(2)}\binom{N}{2}+\bigl(B_r-B_r^{(2)}\bigr)P_2(\mathcal M).\]
Because \(\mathcal M\) is \(r\)-MCS, its character map is injective on \(\mathcal F_r(\Pi)\). Thus the characters occurring on the \(M_r\) flats are precisely \(b_0,\ldots,b_{M_r-1}\), each occurring once. The left-hand side of the preceding identity is therefore
\(\sum_{i=0}^{M_r-1}\binom{b_i}{2}.\)
Rearranging proves \eqref{eq:second-obstruction-general}.

The coefficient \(B_r-B_r^{(2)}\) is positive for \(1\le r\le k-1\), and \(P_2(\mathcal M)\) counts unordered pairs of occurrences supported at the same projective point. Hence the right-hand side must be nonnegative and divisible by \(B_r-B_r^{(2)}\).
\end{proof}

Combining the first- and second-moment identities gives an explicit test at the first-moment endpoint.

\begin{corollary}\label{cor:endpoint-integrality}
Suppose that an \(r\)-MCS projective system has size \(N=N_0\). Then its gap profile satisfies
\(0=\varepsilon_0\le\varepsilon_1\le\cdots\le\varepsilon_{M_r-1}\),
\(\sum_{i=0}^{M_r-1}\varepsilon_i=R_r,\) and
\begin{equation}\label{eq:second-obstruction-profile}
P_2(\mathcal M)=\frac{\displaystyle\binom{M_r}{3}+\sum_{i=0}^{M_r-1}
\left(i\varepsilon_i+\binom{\varepsilon_i}{2}\right)
-B_r^{(2)}\binom{N_0}{2}}{B_r-B_r^{(2)}}.
\end{equation}
In particular, the expression on the right-hand side must be a nonnegative integer.
\end{corollary}

\begin{proof}
At the first-moment endpoint, Proposition~\ref{prop:mcs-first-moment} gives
\(b_i=i+\varepsilon_i.\)
Using
\[\binom{i+\varepsilon_i}{2}=\binom{i}{2}+i\varepsilon_i+\binom{\varepsilon_i}{2},\]
we obtain
\[\sum_{i=0}^{M_r-1}\binom{b_i}{2}=
\sum_{i=0}^{M_r-1}\binom{i}{2}+\sum_{i=0}^{M_r-1}
\left(i\varepsilon_i+\binom{\varepsilon_i}{2}\right).\]
The first sum is
\(\sum_{i=0}^{M_r-1}\binom{i}{2}=\binom{M_r}{3}\)
by the hockey-stick identity. Substitution into Proposition~\ref{prop:second-moment-obstruction} yields \eqref{eq:second-obstruction-profile}.
Since the quotient is exactly \(P_2(\mathcal M)\), it must be a nonnegative integer. Thus the second moment imposes simultaneously a sign condition and a divisibility condition on every admissible gap profile.
\end{proof}

When \(R_r\) is small, the first-moment identity leaves only a limited number of nondecreasing gap profiles. Each can then be tested against the arithmetic condition in
\eqref{eq:second-obstruction-profile}. In favorable cases this rules out the first-moment endpoint entirely and improves the lower bound on \(\nu_r(k,q)\).

\begin{example}\label{ex:binary-4-2}
Consider \((q,k,r)=(2,4,2)\). Then \(M_r=35\), \(B_r=7\), \(B_r^{(2)}=1\), and
\(N_0=\left\lceil\frac{\binom{35}{2}}{7}\right\rceil=85.\)
Here \(R_r=0\). Hence, if a \(2\)-MCS system of size \(85\) existed, Proposition~\ref{prop:mcs-first-moment} would force its character values to be exactly \(0,1,\ldots,34.\)
Then the second-moment identity would give
\[P_2(\mathcal M)=\frac{\binom{35}{3}-\binom{85}{2}}{7-1}=\frac{2975}{6},\]
which is impossible. Therefore no \(2\)-MCS projective system of size \(85\) exists, and \(\nu_2(4,2)\ge86.\)
\end{example}

This example illustrates the role of the defect method. The first moment restricts the possible character-value sets to a small family of near-consecutive configurations, while the second moment tests whether such a configuration is compatible with an actual projective multiset. In this example the arithmetic obstruction strictly improves the first-moment bound. In general, passing the sign and divisibility tests is only necessary and does not guarantee that a character distribution is geometrically realizable.

The elementary interval bound has a rigid equality case. We record it here to complete the lower-bound discussion before constructing systems for the upper bounds.

\begin{proposition}\label{prop:threshold-endpoint}
Let \(1\le r\le k-1\). Then \(\nu_r(k,q)=M_r+r-1\) if and only if \((q,r)=(2,1)\).
\end{proposition}

\begin{proof}
Suppose first that \(\nu_r(k,q)=M_r+r-1\), and let \(\mathcal M\) be a minimum-size \(r\)-MCS projective system. Its \(M_r\) distinct complementary character values all lie in the integer interval \(\{r,r+1,\ldots,M_r+r-1\}\), which itself contains exactly \(M_r\) integers. Hence every value in this interval occurs. Thus \(\mathcal M\) is simultaneously \(r\)-FCS, and its character values are exactly \(0,1,\ldots,M_r-1.\)

Applying the first-moment identity gives \((M_r+r-1)B_r=\binom{M_r}{2}=H_r.\)
Equivalently,
\begin{equation}\label{eq:endpoint-arithmetic}
\frac{M_r}{B_r}=2+\frac{2r}{M_r-1}.
\end{equation}
On the other hand, \(\frac{M_r}{B_r}=\frac{q^k-1}{q^{k-r}-1}>q^r.\)
Since \(k\ge r+1\), monotonicity of Gaussian coefficients gives
\(M_r\ge \qbinom{r+1}{r}=1+q+\cdots+q^r.\)
In particular \(M_r-1\ge2r\), and therefore the right-hand side of \eqref{eq:endpoint-arithmetic} is at most \(3\). It follows that \(q^r<3\), which forces \(q=2\) and \(r=1\).

Conversely, let \(q=2\) and \(r=1\). Choose a vector basis \(e_1,\ldots,e_k\) of \(V\), and assign multiplicity \(2^{i-1}\) to the projective point \([e_i]\). The resulting projective system has total multiplicity \(1+2+\cdots+2^{k-1}=2^k-1=M_1.\)
Every hyperplane of \(\PG(k-1,2)\) is the kernel of a unique nonzero linear functional \(f\in V^\ast\). Its complementary character is \(\sum_{i:f(e_i)=1}2^{i-1}.\)
As \(f\) ranges over the nonzero functionals, the binary vectors \((f(e_1),\ldots,f(e_k))\) range over all nonzero vectors of \(\F_2^k\). Hence the complementary characters are precisely
\(1,2,\ldots,2^k-1.\)
The system is therefore both \(1\)-MCS and \(1\)-FCS, and has size \(M_1=M_1+1-1\).
Thus \(\nu_1(k,2)=M_1\), proving the converse.
\end{proof}

For \(r=1\), the second moment shows that the first-moment endpoint is never attained outside the two classical families.

\begin{corollary}\label{cor:strictly-compact}
Let \(q\ge3\) and \(k\ge3\), and suppose that \(qk\) is even. Then \(q\vartheta_k/2\) is an integer and
\[\nu_1(k,q)\ge\frac{q\vartheta_k}{2}+1.\]
Consequently, \(\nu_1(k,q)=q\vartheta_k/2\) if and only if \(q=2\) or \(k=2\).
\end{corollary}

\begin{proof}
For \(m\ge1\) write \(\vartheta_m=(q^m-1)/(q-1)\), and put \(\vartheta_0=0\). Then \(M_1=\vartheta_k\), \(B_1=\vartheta_{k-1}\), \(B_1^{(2)}=\vartheta_{k-2}\), and \(B_1-B_1^{(2)}=q^{k-2}\). Since \(\vartheta_k-1=q\vartheta_{k-1}\), we have \(H_1/B_1=q\vartheta_k/2\). If \(q\) is even, then \(q\vartheta_k\) is even. If \(q\) is odd, then \(k\) is even, and \(\vartheta_k=\sum_{j=0}^{k-1}q^j\) is a sum of \(k\) odd integers, hence even. Thus \(N_0=q\vartheta_k/2\) and \(R_1=0\).

Suppose that a \(1\)-MCS projective system of size \(N_0\) exists. Since \(R_1=0\), Proposition~\ref{prop:mcs-first-moment} gives \(\lambda=0\) and \(\varepsilon_i=0\) for every \(i\), and Corollary~\ref{cor:endpoint-integrality} gives
\[P_2(\mathcal M)=\frac{X}{q^{k-2}},\quad X=\binom{\vartheta_k}{3}-\vartheta_{k-2}\binom{N_0}{2}.\]
Put \(C=\vartheta_{k-2}\). Then \(\vartheta_{k-1}=1+qC\) and \(\vartheta_k=1+q+q^2C\), and a direct expansion gives
\begin{equation}\label{eq:strictly-compact-X}
24X=q\vartheta_kZ,\quad Z=q^3C^2+5q^2C-7qC+4q+6C-4.
\end{equation}
Write \(q=p^a\) with \(p\) prime, and let \(v_p\) denote the \(p\)-adic valuation. Since \(k\ge3\), we have \(C\equiv1\pmod q\), and hence \(Z\equiv2\pmod q\) and \(\vartheta_k\equiv1\pmod q\). If \(p\ge5\), then \(v_p(Z)=v_p(\vartheta_k)=v_p(24)=0\), so \(v_p(X)=a\); moreover \(q\) is odd, so \(k\ge4\) and \(a<a(k-2)\). If \(p=3\), then \(v_3(Z)=v_3(\vartheta_k)=0\) and \(v_3(24)=1\), so \(v_3(X)=a-1<a(k-2)\). If \(p=2\), then \(q\ge4\), so \(Z\equiv2\pmod4\) and \(\vartheta_k\) is odd; hence \(v_2(X)=a+1-3=a-2<a(k-2)\). In every case \(q^{k-2}\nmid X\), so \(P_2(\mathcal M)\) is not an integer, a contradiction. This proves \(\nu_1(k,q)\ge N_0+1\).

For the last assertion, let \(q,k\ge3\). If \(qk\) is even, the first part applies. If \(qk\) is odd, then \(\vartheta_k\equiv k\equiv1\pmod2\), so \(q\vartheta_k/2\) is not an integer. Conversely, \(\nu_1(k,2)=2^k-1=2\vartheta_k/2\) by Proposition~\ref{prop:threshold-endpoint}, and \(\nu_1(2,q)=q(q+1)/2=q\vartheta_2/2\) by Theorem~\ref{thm:intro-fixed-size}(ii) (see also \cite{AN2019}).
\end{proof}

\begin{remark}\label{rem:mws-33}
For \(q\vartheta_k\) odd, the endpoint \(N_0=(q\vartheta_k+1)/2\) has \(R_1=\vartheta_{k-1}/2>0\), and Corollary~\ref{cor:endpoint-integrality} must be checked profile by profile. For \((q,k)=(3,3)\), we have \(M_1=13\), \(B_1=4\), \(B_1^{(2)}=1\), \(N_0=20\), and \(R_1=2\). The only gap profiles are \(\varepsilon_{12}=2\) and \(\varepsilon_{11}=\varepsilon_{12}=1\), all other gaps being zero, and \eqref{eq:second-obstruction-profile} gives \(P_2(\mathcal M)=121/3\) and \(P_2(\mathcal M)=119/3\), respectively. Hence \(\nu_1(3,3)\ge21\). On the other hand, the same test does not exclude \(N_0\) for \((q,k)=(5,3)\), \((7,3)\), \((9,3)\), or \((3,5)\).
\end{remark}

\subsection{Upper bounds for the \(r\)-MCS threshold}
\label{subsec:threshold-upper}

We now turn to upper bounds for the \(r\)-MCS threshold. For a prescribed separating support \(S\), the parameter \(\tau_r(S)\) is the minimum positive-labeling cost. We bound this cost and apply the result to the full point set and to the complement of one flat.

For an \(r\)-separating set \(S\subseteq\Pi\) and positive integers \(a_P\), write
\[\chi_a(\Lambda)=\sum_{P\in S\cap\Lambda}a_P,\quad \Lambda\in\mathcal F_r(\Pi).\]
By Proposition~\ref{prop:separating-spans}, \(S\) spans \(\Pi\). Thus an injective map \(\chi_a\) defines an \(r\)-MCS system of total multiplicity \(\sum_{P\in S}a_P\).

A \emph{finite hypergraph} is a pair \(\mathcal H=(X,\mathcal E)\), where \(X\) is a finite set of \emph{vertices} and \(\mathcal E\) is a family of subsets of \(X\), called \emph{hyperedges}.
For a labeling \(a:X\to\Z_{\ge0}\), the \emph{weight} of an edge \(E\in\mathcal E\) is \(\omega_a(E)=\sum_{P\in E}a(P)\). The labeling is an \emph{edge discriminator} if the weights of all edges are positive and pairwise distinct. We use this notion only for hypergraphs with distinct, nonempty edges.
We recall the initial-label version of the edge-discrimination lemma of Bhattacharya, Das, and Ganguly \cite[Lemma~1]{BDG2014}, using the notation needed below.

\begin{lemma}[Bhattacharya-Das-Ganguly]\label{BDG}
Let \(\mathcal H=(X,\mathcal E)\) be a finite hypergraph with \(h\) distinct nonempty edges, and order its vertices as \(X=\{P_1,\ldots,P_u\}\). For \(1\le i\le u\), let
\[c_i=\bigl|\{E\in\mathcal E:P_i\in E,\ E\subseteq\{P_1,\ldots,P_i\}\}\bigr|.\]
Given any initial labeling \(\kappa:X\to\Z_{\ge0}\), there exists an edge discriminator \(a:X\to\Z_{\ge0}\) such that \(a(P)\ge\kappa(P)\) for every \(P\in X\) and
\begin{equation}\label{eq:BDG-initial-labeling}
\sum_{P\in X}a(P)
\le \binom{h+1}{2}
-\sum_{\substack{1\le i\le u\\\kappa(P_i)>0}}c_i
+\sum_{P\in X}\kappa(P).
\end{equation}
\end{lemma}

We now apply Lemma~\ref{BDG} to projective separators. Let \(S\subseteq\Pi\) be a \(r\)-separating set, and consider its \emph{trace hypergraph} \(\mathcal H_r(S)=(S,\mathcal E_r(S))\), where
\[\mathcal E_r(S)=\{S\cap\Lambda:\Lambda\in\mathcal F_r(\Pi),\
S\cap\Lambda\ne\varnothing\},\qquad
h=|\mathcal E_r(S)|.\]
Its vertices are the projective points of \(S\), and its edges are the nonempty traces of codimension-\(r\) flats. Since \(S\) is \(r\)-separating, all \(M_r\) traces are distinct and at most one is empty. Thus \(h\in\{M_r-1,M_r\}\), and the trace hypergraph has no repeated or empty edges. For each nonempty trace, \(\chi_a(\Lambda)=\omega_a(S\cap\Lambda)\).

\begin{theorem}\label{thm:positive-labeling}
Let \(1\le r\le k-1\), let \(S\subseteq\Pi\) be a \(r\)-separating set, and put \(h=|\mathcal E_r(S)|\). There exist positive integer labels \(a_P\), \(P\in S\), such that the weighted traces \(\chi_a(\Lambda)\), \(\Lambda\in\mathcal F_r(\Pi)\), are pairwise distinct and
\[\sum_{P\in S}a_P\le |S|+\binom h2.\]
Consequently,
\begin{equation}\label{eq:total-labeling-cost}
\tau_r(S)\le |S|+\binom h2\le |S|+H_r.
\end{equation}
The same labeling satisfies \(1\le a_P\le1+\binom h2\le H_r+1\) for every \(P\in S\).
\end{theorem}

\begin{proof}
Apply Lemma~\ref{BDG} to \(\mathcal H_r(S)\) with the constant initial labeling \(\kappa(P)=1\) for every \(P\in S\). Every nonempty edge has exactly one largest-index vertex, so \(\sum_{i=1}^{|S|}c_i=h\). Since all initial labels are positive, the bound \eqref{eq:BDG-initial-labeling} becomes
\[\sum_{P\in S}a_P\le\binom{h+1}{2}-h+|S|=|S|+\binom h2.\]
The inequality \(a_P\ge\kappa(P)=1\) guarantees that all labels are positive. Thus they define a projective multiset whose positive support is precisely \(S\), and which therefore spans \(\Pi\). By Lemma~\ref{BDG}, the nonempty traces have positive, pairwise distinct weights. If an empty trace occurs, its weight is zero; since there is at most one empty trace, this value is distinct from all other weights. Hence \(\chi_a\) is injective on \(\mathcal F_r(\Pi)\), and the constructed labeling is admissible for \(\tau_r(S)\). Using \(h\le M_r\) now proves \eqref{eq:total-labeling-cost}.

Finally, the other \(|S|-1\) labels contribute at least \(|S|-1\), so for every \(P\in S\) we have \(a_P\le 1+\binom h2\le H_r+1\), as claimed.
\end{proof}

The initial labeling \(\kappa\equiv1\) is important. Adding one to arbitrary nonnegative labels changes a trace sum by its cardinality and can destroy injectivity when trace sizes differ. Similarly, subtracting one from positive labels is not valid on an arbitrary fixed support. The first construction below permits this subtraction because all flats have the same number of points; the second instead uses a support with exactly one empty trace.

\begin{proof}[Proof of Theorem~\ref{thm:intro-threshold}]
\emph{(i)} Let \(\mathcal M\) be an \(r\)-MCS system of minimum size, and put \(S=\supp(\mathcal M)\). If two different flats had the same trace on \(S\), their character sums would be equal. Thus \(S\) is \(r\)-separating, and its multiplicities form an admissible labeling of cost \(\nu_r(k,q)\). It follows that \(\min_S\tau_r(S)\le\nu_r(k,q)\).
Conversely, let \(S\) be any \(r\)-separating set and choose a labeling attaining \(\tau_r(S)\). Such a labeling exists by Definition~\ref{def:separating-labeling} and the observation following it. The labels define an \(r\)-MCS projective system because \(S\) spans \(\Pi\) by Proposition~\ref{prop:separating-spans}. Hence \(\nu_r(k,q)\le\tau_r(S)\). Taking the minimum proves part~\textup{(i)}.

\smallskip
\emph{(ii)} The lower bound is Proposition~\ref{prop:mcs-first-moment}. We establish the upper bounds in two steps.
First take \(S=\Pi\). All \(M_r\) flats are nonempty, so Theorem~\ref{thm:positive-labeling} gives positive labels \(a_P\) with distinct flat sums and
\(\sum_{P\in\Pi}a_P\le\vartheta_k+H_r.\)
Set \(b_P=a_P-1\ge0\). Every codimension-\(r\) flat contains \(\theta_q(k-r-1)\) points. Therefore,
\[\sum_{P\in\Lambda}b_P=\sum_{P\in\Lambda}a_P-\theta_q(k-r-1),
\quad \Lambda\in\mathcal F_r(\Pi),\]
so all character sums remain distinct. The set \(S_b=\{P:b_P>0\}\) is consequently \(r\)-separating and spans \(\Pi\) by Proposition~\ref{prop:separating-spans}. Discarding zero labels gives an \(r\)-MCS system with total multiplicity at most \(H_r\). Thus
\begin{equation}\label{eq:uniform-quadratic-upper}
\nu_r(k,q)\le H_r.
\end{equation}

For the sharper bound, fix \(\Lambda_0=\PG(W_0)\in\mathcal F_r(\Pi)\) and put \(S_0=\Pi\setminus\Lambda_0\). We show that \(S_0\) is \(r\)-separating and that \(\Lambda_0\) is its unique empty trace. Write \(s=k-r\). If \(\Lambda=\PG(U)\ne\Lambda_0\), then \(U\cap W_0\) is a proper subspace of \(U\). The vectors of \(U\setminus W_0\) span \(U\). Indeed, choose \(x\in U\setminus W_0\). For every \(y\in U\cap W_0\), both \(x\) and \(x+y\) lie in \(U\setminus W_0\), and \(y=(x+y)-x\). Hence, we have \(\langle S_0\cap\Lambda\rangle
=\langle\Lambda\setminus\Lambda_0\rangle=\Lambda\) for \(\Lambda\ne\Lambda_0\).
These traces are nonempty and distinct, since each determines its spanning flat. The trace of \(\Lambda_0\) is empty. Thus \(S_0\) is \(r\)-separating with \(h=M_r-1\), and
\(|S_0|=\vartheta_k-\theta_q(s-1)=\frac{q^k-q^{k-r}}{q-1}.\)
By Theorem~\ref{thm:positive-labeling}, we have
\begin{equation}\label{eq:complement-flat-upper}
\nu_r(k,q)\le\tau_r(S_0)
\le \binom{M_r-1}{2}+\frac{q^k-q^{k-r}}{q-1}.
\end{equation}

It remains to compare this estimate with \(H_r\). We have \(M_r\ge\vartheta_k\). To see this, put \(d=\min\{r,k-r\}\) and use symmetry of Gaussian coefficients. For \(1\le j<d\), \(\qbinom{k}{j+1}/{\qbinom{k}{j}}=\frac{q^{k-j}-1}{q^{j+1}-1}\ge1\)
since \(d\le k/2\). Hence \(M_r=\qbinom{k}{d}\ge\qbinom{k}{1}=\vartheta_k\). Consequently,
\[
\binom{M_r-1}{2}+\vartheta_k-\theta_q(s-1)
=H_r-(M_r-\vartheta_k)-(\theta_q(s-1)-1)
\le H_r.
\]
When \(s\ge2\), the last inequality is strict because \(\theta_q(s-1)>1\). When \(s=1\), equivalently \(r=k-1\), it is an equality, and Theorem~\ref{thm:intro-fixed-size}\textup{(ii)} gives \(\nu_{k-1}(k,q)=H_{k-1}\). This proves part~\textup{(ii)} and the equality assertion.
\end{proof}

For \(r=1\), Theorem~\ref{thm:intro-threshold} concerns the maximum weight spectrum codes of \cite{SZSC,AN2019}, that is, \([n,k]_q\) codes with \(\vartheta_k\) distinct nonzero weights. Let \(n_{\mathrm{MWS}}(k,q)\) denote the minimum length of such a code.

\begin{corollary}[Maximum weight spectrum codes]\label{cor:mws-length}
Let \(k\ge2\). Then \(n_{\mathrm{MWS}}(k,q)=\nu_1(k,q)\) and
\begin{equation}\label{eq:mws-length}
\left\lceil\frac{q\vartheta_k}{2}\right\rceil\le n_{\mathrm{MWS}}(k,q)\le\binom{\vartheta_k-1}{2}+q^{k-1}\le3q^{2k-2}.
\end{equation}
In particular, \(\limsup_{k\to\infty}\frac1k\log_qn_{\mathrm{MWS}}(k,q)\le2\) for every fixed prime power \(q\).
\end{corollary}

\begin{proof}
A code \(C\) is a maximum weight spectrum code if and only if \(|\{\wt(D):D\le C,\ \dim D=1\}|=\vartheta_k\). Puncturing the identically zero coordinates of \(C\) preserves this property and does not increase the length, so the minimum length is attained by a nondegenerate code; by Proposition~\ref{prop:char-r}, nondegenerate maximum weight spectrum codes of length \(n\) correspond to \(1\)-MCS projective \([n,k]\) systems. Hence \(n_{\mathrm{MWS}}(k,q)=\nu_1(k,q)\).
Now take \(r=1\) in Theorem~\ref{thm:intro-threshold}(ii). Here \(M_1=\vartheta_k\), \(B_1=\vartheta_{k-1}\), and \((q^k-q^{k-1})/(q-1)=q^{k-1}\). Since \(\vartheta_k-1=q\vartheta_{k-1}\), the lower bound \(\lceil H_1/B_1\rceil\) equals \(\lceil q\vartheta_k/2\rceil\). Finally, \(\vartheta_k\le2q^{k-1}\) gives \(\binom{\vartheta_k-1}{2}+q^{k-1}\le2q^{2k-2}+q^{k-1}\le3q^{2k-2}\), and the asymptotic statement follows.
\end{proof}

\begin{remark}\label{rem:mws-comparison}
The lower bound in \eqref{eq:mws-length} is \cite[Lemma~5.1]{AN2019}. For \(k=2\), the upper bound equals \(q(q+1)/2\), which is the exact value of \(n_{\mathrm{MWS}}(2,q)\) \cite{AN2019}, in accordance with Theorem~\ref{thm:intro-fixed-size}(ii). For \(k\ge3\), the previously known upper bounds on \(n_{\mathrm{MWS}}(k,q)\) are, to our knowledge, the following: the constructions of \cite{AN2019} have length growing like \(q^{\Theta(k^2)}\) (see \cite[Proposition~3.2]{A2019}); the construction of \cite[Proposition~3.3]{A2019} has length \(\theta_q(k-2)\binom{\vartheta_k}{2}\), which is of order \(q^{3k-4}/2\); and the averaging argument of \cite[Theorem~1]{CT2018} gives \(n_{\mathrm{MWS}}(k,q)\le D_qq^{4k}\) with a constant \(D_q\) depending only on \(q\). The bound \eqref{eq:mws-length} is of order \(q^{2k-2}/2\), and the labeling of Lemma~\ref{BDG} is algorithmic. Cohen and Tolhuizen \cite[Section~5]{CT2018} asked whether \(\lim_{k\to\infty}\frac1k\log_qn_{\mathrm{MWS}}(k,q)\) exists, observing that the lower and upper limits lie in the interval \([1,4]\); Corollary~\ref{cor:mws-length} reduces the upper limit to \(2\).
On the lower side, Corollary~\ref{cor:strictly-compact} and Remark~\ref{rem:mws-33} improve the bound \(\lceil q\vartheta_k/2\rceil\) by one whenever \(q,k\ge3\) and \(qk\) is even, and for \((q,k)=(3,3)\). In the terminology of \cite{MK2020}, strictly compact maximum weight spectrum codes, that is, codes of length \(q\vartheta_k/2\), exist if and only if \(q=2\) or \(k=2\). In particular, the hypothetical \([42,3,22]_4\) code discussed in \cite[Example~11]{MK2020} does not exist, and \(n_{\mathrm{MWS}}(3,3)\ge21\), whereas \cite{MK2020} gives \(n_{\mathrm{MWS}}(3,3)\ge20\).
\end{remark}

The threshold bounds have the following asymptotic consequences.
\begin{corollary}\label{cor:threshold-asymptotics}
\begin{enumerate}
\item[(i)] If \(q\) and \(r\) are fixed and \(k\to\infty\), then
\[
\nu_r(k,q)=\Omega_{q,r}(M_r),
\quad \nu_r(k,q)=O(M_r^2).
\]
\item[(ii)] If \(k,r\) are fixed, with \(1\le r\le k-1\), and \(q\to\infty\) through prime powers, then
\[
\nu_r(k,q)=\Omega_{k,r}\bigl(q^{r(k-r+1)}\bigr),
\quad \nu_r(k,q)=O_{k,r}\bigl(q^{2r(k-r)}\bigr).
\]
\end{enumerate}
\end{corollary}

\begin{proof}
For part~\textup{(i)}, the ratio
\(\frac{M_r}{B_r}=\frac{q^k-1}{q^{k-r}-1}\)
tends to \(q^r\). Therefore \(H_r/B_r=\Theta_{q,r}(M_r)\), and the lower bound follows from Theorem~\ref{thm:intro-threshold}. Its upper bound is at most \(H_r=O(M_r^2)\).

For part~\textup{(ii)}, by the Gaussian product formula, we have \(M_r=\Theta_{k,r}\bigl(q^{r(k-r)}\bigr)\) and \(B_r=\Theta_{k,r}\bigl(q^{r(k-r-1)}\bigr).\)
Thus \(H_r/B_r=\Theta_{k,r}(q^{r(k-r+1)})\), while \(H_r=\Theta_{k,r}(q^{2r(k-r)})\). Applying the two bounds in Theorem~\ref{thm:intro-threshold} proves the assertion.
\end{proof}
The complement-of-a-flat construction improves the explicit upper bound but not these asymptotic exponents. The exact cases \(r=k-1\) and \((q,r)=(2,1)\) are given by Theorem~\ref{thm:intro-fixed-size}\textup{(ii)} and Proposition~\ref{prop:threshold-endpoint}, respectively.

\section{Conclusion and open problems}\label{sec:open}

In this paper, we have developed an extremal theory for character spectra of finite projective systems. We establish the existence of \(r\)-FCS systems in a broad parameter range, and solve the size-constrained extremal problem completely when \(r=k-1\), including the minimum-cost multiplicity profiles. We also develop global moment identities and local flat-incidence bounds that impose strong structural constraints on large character spectra. 
For the \(r\)-MCS threshold, the first moment gives \(\nu_r(k,q)\ge\lceil H_r/B_r\rceil\), and the second moment can exclude the rounded endpoint. In the other direction, a positive additive-labeling theorem converts distinct traces into distinct integer character values. The full point set yields \(\nu_r(k,q)\le H_r\) after uniform subtraction of one from each label. Taking the complement of one flat gives the stronger explicit upper bound in Theorem~\ref{thm:intro-threshold}, with a strict improvement whenever \(r\le k-2\).

The case \(r=1\) deserves particular mention, since \(\nu_1(k,q)\) is the minimum length \(n_{\mathrm{MWS}}(k,q)\) of a maximum weight spectrum code (Corollary~\ref{cor:mws-length}). Here both directions of our method give new information. On the upper side, the complement-of-a-flat construction gives \(n_{\mathrm{MWS}}(k,q)\le\binom{\vartheta_k-1}{2}+q^{k-1}\), which is of order \(q^{2k-2}/2\). The previously known bounds were of orders \(q^{\Theta(k^2)}\) \cite{AN2019}, \(q^{3k-4}/2\) \cite[Proposition~3.3]{A2019}, and \(D_qq^{4k}\) \cite{CT2018}, so the lower and upper limits of \(\frac1k\log_qn_{\mathrm{MWS}}(k,q)\) in the problem of Cohen and Tolhuizen \cite{CT2018} now lie in \([1,2]\) instead of \([1,4]\). On the lower side, the second-moment obstruction of Proposition~\ref{prop:second-moment-obstruction} and Corollary~\ref{cor:endpoint-integrality} improves the first-moment bound \(\lceil q\vartheta_k/2\rceil\) of \cite{AN2019,MK2020} whenever \(q,k\ge3\) and \(qk\) is even (Corollary~\ref{cor:strictly-compact}). As a consequence, strictly compact maximum weight spectrum codes exist only in the classical cases \(q=2\) and \(k=2\). The same obstruction applies verbatim to every codimension \(r\), as Example~\ref{ex:binary-4-2} illustrates for \(r=2\).

The remaining questions concern the full-spectrum range, intermediate spectrum sizes, and the gap between the threshold bounds.

\begin{problem}\label{prob:complete-threshold}
Prove or disprove Conjecture~\ref{conj:rFCS} in the remaining range \(2\le r\le k-2\). Equivalently, determine whether an \(r\)-FCS projective system of rank \(k\) and total multiplicity \(n\) can exist when \(n-r+1\ge 2^{k-r+1}.\)
In particular, can such a system exist outside the common-flat configuration described in Proposition~\ref{prop:common-flat}?
\end{problem}

\begin{problem}\label{prob:intermediate-function}
Determine the extremal function \(L_r(n,k,q)\) in the intermediate range between the known \(r\)-FCS regime and the stable \(r\)-MCS regime, especially for \(1<r<k-1\). More generally, characterize the finite sets of integers that can occur as \(\Sigma_r(\mathcal M)\) for spanning projective multisets \(\mathcal M\) of total multiplicity at most \(n\).
\end{problem}

\begin{problem}\label{prob:threshold-gap}
Beyond the exact cases \(r=k-1\) and \((q,r)=(2,1)\), improve the bounds
\[
\left\lceil\frac{H_r}{B_r}\right\rceil
\le\nu_r(k,q)
\le\binom{M_r-1}{2}+\frac{q^k-q^{k-r}}{q-1}.
\]
In particular, determine the asymptotic order of \(\nu_r(k,q)\) in the two regimes of Corollary~\ref{cor:threshold-asymptotics}. Can the higher moments improve the lower bound, or can the point--flat incidence structure be used to reduce the total additive-labeling cost and improve the upper-bound exponents?
\end{problem}

These questions form an integral realizability problem for the point--flat incidence map. The Smith normal forms computed by Chandler, Sin, and Xiang \cite{CSX2006} describe the corresponding integral lattice. Requiring nonnegative weights with spanning support and bounded total multiplicity adds further constraints. The moment identities provide necessary conditions, while the labeling constructions provide sufficient conditions for maximum spectra. By Proposition~\ref{prop:char-r}, the same questions describe the support-weight spectra of \(r\)-dimensional subcodes.

\paragraph{Acknowledgments.} During the preparation and revision of this manuscript, the authors used ChatGPT (OpenAI, including GPT-6) for English-language editing, manuscript organization, and preliminary checks of proofs. ChatGPT also assisted in the exploration and development of the improved upper bound in Theorem 1.5(ii). The authors take full responsibility for the mathematical statements, proofs, originality claims, and references in the manuscript.
The research of Shitao Li was supported by the National Natural Science Foundation of China under Grant Nos. 12526612 and 126011020. The research of Yang Li was supported by the Nanyang Technological University Research under Grant 04INS000047C230GRT01. The research of Gaojun Luo was supported by the National Natural Science Foundation of China under Grant 12401690.

\end{document}